\documentclass[12pt]{amsart}

\usepackage{tikz}
\usepackage[a4paper]{geometry}
\usepackage{amsmath,amssymb,amsthm}
\usepackage{bbm}
\usepackage{todonotes}

\usepackage{silence}
\usetikzlibrary{matrix,arrows}
\usepackage[hyphens]{url}
\usepackage[backend=biber]{biblatex}
\usepackage{hyperref}
\usepackage{microtype}
\usepackage[utf8]{inputenc}

\newtheorem{lem}{Lemma}[section]
\newtheorem{prop}[lem]{Proposition}
\newtheorem{thm}[lem]{Theorem}

\newtheorem{korl}[lem]{Corollary}

\theoremstyle{definition}
\newtheorem{dfn}[lem]{Definition}
\newtheorem{ex}[lem]{Example}

\theoremstyle{remark}

\newcommand{\Z}{\mathbb{Z}}
\newcommand{\R}{\mathbb{R}}
\newcommand{\Q}{\mathbb{Q}}
\newcommand{\N}{\mathbb{N}}
\newcommand{\C}{\mathbb{C}}
\newcommand{\id}{\mathrm{id}}

\DeclareMathOperator{\End}{End}

\DeclareMathOperator{\diam}{diam}
\DeclareMathOperator{\hc}{hc}
\DeclareMathOperator{\Fr}{Fr}

\DeclareMathOperator{\dist}{dist}

\title{Almost flat bundles and homological invariance of infinite K-area}
\author[B.\ Hunger]{Benedikt Hunger}
\email{benedikt.seelstrang@math.uni-augsburg.de}

\begin{document}

\begin{abstract}
  We extend the notion of an almost flat bundle over a closed Riemannian manifold to bundles over simplicial complexes, and prove 
  that up to a constant factor, this notion is invariant under pullback via maps which induce isomorphisms on fundamental groups. 
  As an application, we show that the property of having infinite K-area only depends on the image of the fundamental class under 
  the classifying map of the universal cover.
\end{abstract}

\maketitle 

\section{Introduction and statement of results}

Connes, Gromov and Moscovici \cite{connes-gromov-moscovici-conjecture-de-novikov} introduced the notion of almost flat K-theory
classes in order to give a unified approach to different special cases of the Novikov conjecture. We first recall the definition
of an asymptotically flat K-theory class\footnote{The term in \cite{connes-gromov-moscovici-conjecture-de-novikov} is ``fibré
presque plat'', while the standard term is ``almost flat K-theory class''. However, following Manuilov
and Mishchenko \cite{manuilov-mishchenko-almost-asymptotic-fredholm-representations}, we will use the word \emph{almost} for
notions depending on a parameter $\epsilon>0$, and \emph{asymptotic} for any kind of limit as $\epsilon\to 0$.}\ as introduced by
Connes, Gromov and Moscovici.

Let $M$ be a differentiable manifold. The \emph{curvature} of a smooth vector bundle $E\to M$ with connection $\nabla$ is the
endomorphism-valued 2-form $\mathcal R^\nabla\in\Omega^2(M;\End(E))$ given by the formula
\[
  \mathcal R^\nabla(X,Y)s=\nabla_X\nabla_Ys-\nabla_Y\nabla_Xs-\nabla_{[X,Y]}s.
\]
Suppose in addition that $E$ is a Hermitian vector bundle (i.~e. a complex vector bundle where the fibres are equipped with a 
smoothly varying Hermitian inner product) and the connection $\nabla$ is compatible with the metric (i.~e. parallel transport is
an isometry). Then $\End(E)$ is a bundle of normed spaces where the norm over each fibre is simply
the operator norm. If, additionally, $M$ is a Riemannian manifold with metric $g$, the norm of the curvature of $E$ is defined by
\[
  \|\mathcal R^\nabla\|_g=\sup_{\substack{X\wedge Y\in\Lambda^2TM\\\|X\wedge Y\|_g\leq 1}}\|\mathcal R^\nabla(X,Y)\|_{\text{op}}.
\]
Here the norm on $\Lambda^2TM$ is given by $\|X\wedge Y\|_g^2=\|X\|^2_g\|Y\|^2_g-g(X,Y)^2$.

Now a class $\eta\in K^0(M)$ is \emph{asymptotically flat} if there are sequences $(E_n,\nabla_n)$ and 
$(\tilde E_n,\tilde\nabla_n)$ of Hermitian vector bundles with compatible connecion over $M$, such that $\eta=[E_n]-[\tilde E_n]$ 
for every $n\in\N$, and such that
\[
  \lim_{n\to\infty}\|\mathcal R^{\nabla_n}\|_g=\lim_{n\to\infty}\|\mathcal R^{\tilde\nabla_n}\|_g=0.
\]

A priori, this notion of asymptotic flatness depends on the choice of metric on $M$. However, if $M$ is compact, then any two
metrics are bi-Lipschitz equivalent. Thus, given metrics $g$ and $\tilde g$ on a compact manifold $M$, there is a constant $c>0$
such that
\[
  \|\mathcal R^\nabla\|_g\leq c\|\mathcal R^\nabla\|_{\tilde g}
\]
for all Hermitian bundles $(E,\nabla)$ with compatible connection. In particular, K-theory classes are asymptotically flat with 
respect to $g$ if and only if they are almost flat with respect to $\tilde g$.

Now suppose that $(E,\nabla)$ is a Hermitian vector bundle with compatible connection over $M$ which satisfies 
$\|\mathcal R^\nabla\|_g\leq\epsilon$. We will call such a bundle an
\emph{$\epsilon$-flat bundle} (or an $\epsilon$-almost flat bundle) over $M$. A fundamental result from classical Riemannian 
geometry states that if we are given a
bundle with small curvature then parallel transport along a nullhomotopic curve $\gamma$ is $c\epsilon$-close to the identity, 
where $c$ is a constant depending only on $M$ and $\gamma$. In fact, $c$ may be chosen as the area of a disk filling $\gamma$.

We may use this to construct a map $f\colon\pi_1(M,x_0)\to U(l)$ as follows: We fix a trivialization of the fibre over the 
base-point $x_0$. For each class $c\in\pi_1(M,x_0)$, we choose a representing
loop $\gamma$. We define $f(c)$ to be the parallel transport along $\gamma$ with respect to the
trivialization of the fibre over the base-point. Then $f$ might not be a group homomophism, but 
$\|f(gh)-f(g)f(h)\|\leq c(g,h)\epsilon$, where $c(g,h)>0$ is a constant
depending only on $g$ and $h$, but not on the bundle $E$. Such data constitute a \emph{quasi-representation} of $M$, and in fact,
almost flat bundles and quasi-representations of the fundamental group turn out to be two sides of the same coin. This
relationship between almost flat bundles and quasi-representations was already noted by Connes, Gromov and Moscovici
\cite{connes-gromov-moscovici-conjecture-de-novikov} and was made precise by Carrión and Dadarlat
\cite{carrion-dadarlat-almost-flat-k-theory}. However, their exposition seems to be a bit ad hoc, while similar results will be 
natural consequences of the results presented in this paper.

There is another important consequence of the fact that parallel transport along contractible curves is close to the identity.
Namely, suppose that $M$ is smoothly triangulated, and $\sigma$ is a simplex of $M$. Then, after a choice of basis for the fibre
over the barycentre of $\sigma$, we can trivialize the $E$ over $\sigma$ by parallel transporting this basis from the barycenter
outwards. Now if $\rho$ is another simplex then the transition functions between those two trivializations turn out to be
Lipschitz functions, where the Lipschitz constant is small if $\epsilon$ is small. 

This idea enabled Mishchenko and Teleman
\cite{mishchenko-teleman-almost-flat-bundles} to show that every $\epsilon$-flat bundle can be pulled back from a bundle over
$B\pi_1(M)$ along the classifying map of the universal cover of $M$ if $\epsilon$ is sufficiently small. In the course of the
proof of this statement, they introduced the concept of small bundles over a simplicial complex: An
$\epsilon$-flat bundle is a vector bundle such that the transition functions with respect to some family of trivializations over 
the simplices are $\epsilon$-Lipschitz. We will adapt to this definition in this paper and show that it is, in fact, more or less 
equivalent to the old one if we consider bundles over a triangulated manifold. Also, our results will prove a generalization of
the theorem of Mishchenko and Teleman, namely that every $\epsilon$-flat bundle is the pull-back of a $c\epsilon$-flat bundle over
a finite subcomplex of $B\pi_1(M)$ if $\epsilon$ is small enough.

Dual to the concept of almost flat K-theory classes is the notion of \emph{infinite K-area} introduced by Gromov
\cite{gromov-positive-curvature-macroscopic-dimension}. Namely, a Riemannian manifold $M$ has infinite K-area if, for every
$\epsilon>0$, there is an $\epsilon$-flat bundle $(E,\nabla)$ over $M$ with at least one non-vanishing Chern number. Infinite
K-area is one of several important largeness properties of Riemannian manifolds introduced by Gromov and Lawson
\cite{gromov-large-riemannian-manifolds,gromov-lawson-spin-scalar-curvature}. Brunnbauer and Hanke 
\cite{brunnbauer-hanke-large-and-small} showed that other largeness properties, including enlargeability, are 
\emph{homologically invariant} in the sense that they only depend on 
the image of the fundamental class under the classifying map of the universal cover. Their proof proceeds as follows: First they 
define enlargeability of an arbitrary homology
class of a simplicial complex in such a way that a closed Riemannian manifold $M$ is enlargeable if and only if its fundamental
class $[M]$ is enlargeable. Then they use an extension lemma to show that if a map $f\colon X\to Y$ induces an isomorphism on
fundamental groups, then a class $\eta\in H_*(X)$ is enlargeable if and only if $f_*\eta\in H_*(Y)$ is. Since, by definition, the
classifying map of the universal cover $\Phi\colon M\to B\pi_1(M)$ induces an isomorphism of fundamental groups, this implies 
homological invariance of enlargeability.

Motivated by this scheme, we will define a homology class to have infinite K-area if, for every $\epsilon>0$, there is an 
$\epsilon$-flat bundle whose Chern classes detect the given homology class. An extension result for $\epsilon$-flat bundles will
then be used to show that infinite K-area is homologically invariant. This also implies that the infinite K-area
is invariant under $p$-surgery with $p\neq 1$, a fact which has been proven directly by Fukumoto 
\cite{fukumoto-invariance-karea-surgery}.

Hanke and Schick \cite{hanke-positive-scalar-curvature,hanke-schick-novikov-low-degree-cohomology} used a notion of almost flat
bundles of Hilbert-$A$-modules for arbitrary C*-algebra to prove a special case of the Novikov conjecture. It turns out that one
needs precisely the Lipschitz condition on the transition functions in order to prove their results. Therefore, it makes sense 
not only to consider Hermitian vector bundles, but also bundles of Hilbert $A$-modules for arbitrary C*-algebras $A$. 

We conclude the introduction by giving an outline of the following sections and the main results. 

In section \ref{sec:almost-flat-bundles} we will give the precise definition of an $\epsilon$-flat bundle of Hilbert $A$-modules,
and show that examples are given by Hilbert module bundles with compatible connection having small curvature.

Section \ref{sec:trivialization-lemma} provides the most important technical result of this paper, the \emph{trivialization
lemma} \ref{thm:trivialization lemma}. This states that every $\epsilon$-flat bundle over a simply-connected space is trivial if
$\epsilon$ is small enough. As a corollary, one can extend almost flat bundles defined on the boundary of a disk $D^k$ to the
whole disk, since they are trivial on the boundary. The main ingredients in the proof of the trivialization lemma are an extension
statement for unitary-valued Lipschitz functions (lemma \ref{lem:extension-of-unitary-valued-lipschitz-maps}) and a
combinatorial version of the statement that parallel transport along boundaries of small disks is close to the identity (theorem
\ref{thm:transport-along-contractible-loops}).

In section \ref{sec:first-applications}, we will give first applications of the trivialization lemma: Firstly, we show that almost 
flat bundles on the barycentric subdivision of a finite-dimensional complex are almost flat with respect to the original complex.
Secondly, we give conditions under which an almost flat bundle can be extended to an almost flat bundle over a larger subcomplex.

The rest of this paper will consist of rather easy applications of the trivialization lemma, beginning with section
\ref{sec:almost-representations}, where we relate the concepts of almost representations and quasi-representations to the concept
of almost flat K-theory classes.

In section \ref{sec:riemannian-manifolds} we will use this to show that our definition of an almost flat bundle corresponds to 
the definition via smooth connections of Connes, Gromov and Moscovici \cite{connes-gromov-moscovici-conjecture-de-novikov}. This 
will make use of an extension theorem for connections with small curvature which is mainly due to Fukumoto 
\cite{fukumoto-invariance-karea-surgery} in his proof of invariance of infinite K-area under certain surgeries. This will be used
later in order to show that our definition of infinite K-area is a generalization of Gromov's
\cite{gromov-positive-curvature-macroscopic-dimension} definition.

Next, in section \ref{sec:pullbacks}, we use our results on extension of almost flat bundles in order to show the following
functoriality result: Given a map $f\colon X\to Y$, almost flat bundles over $Y$ pull back to
almost flat bundles over $X$, and if the map $f$ induces an isomorphism on fundamental groups, the pull-back map is in fact
surjective (in a certain sense) on almost flat bundles. This gives the generalization of the theorem of Mishchenko and Teleman
\cite{mishchenko-teleman-almost-flat-bundles} cited before. This section is independent of section \ref{sec:riemannian-manifolds}.

We will put all those results together in section \ref{sec:homological-invariance} to define a notion of infinite K-area for 
arbitrary homology classes of simplicial complexes, and to prove homological invariance, i.~e. a class has infinite K-area if and 
only if its image under the classifying map of the universal cover has infinite K-area. This will directly imply that, for a 
Riemannian manifold, having infinite K-area only depends on the image of the fundamental class under the classifying map of the 
universal cover. We will show how to use this to regather the theorem of Fukumoto \cite{fukumoto-invariance-karea-surgery} about 
the invariance of infinite K-area under surgeries in codimension not equal to one.

Finally, in section \ref{sec:almost-representations}, we relate the notions of almost and asymptotic representations
\cite{manuilov-mishchenko-almost-asymptotic-fredholm-representations}, of almost flat bundles, and of asymptotically flat
K-theory. This section only uses material from sections \ref{sec:trivialization-lemma} and \ref{sec:first-applications}, and may
be read independently of the rest of this paper.

This work is based on the author's Master's thesis at Universität Augsburg. The author would like to thank his thesis advisor
Bernhard Hanke for his invaluable help and support. 

\section{Almost flat bundles}

\label{sec:almost-flat-bundles}

\subsection{Preliminaries and main definition}

The principal aim of this section is to give the definition of an almost flat bundle over an arbitrary simplicial complex. 
It seems to be useful to consider not only Hermitian bundles, but rather bundles of Hilbert C*-modules
\cite[cf.]{hanke-positive-scalar-curvature,schick-l2-index-kk-theory-connections}. We will work in this more general setting,
since it does not require any more work. 

Let $A$ be a C*-algebra. Recall that a \emph{Hilbert $A$-module} is a right $A$-module $V$ together with an inner product 
$V\times V\to A$, $(v,w)\mapsto\langle v,w\rangle$, satisfying certain conditions \cite{jensen-thomsen-elements-kk-theory}, for 
instance that $\|v\|=\sqrt{\|\langle v,v\rangle\|}$ defines a complete norm on $V$. For example, a Hilbert $\C$-module is the same
thing as a complex Hilbert space.

Given two Hilbert $A$-modules $V$ and $W$, a 
map $f\colon V\to W$ is called \emph{adjointable} if there is another map $f^*\colon W\to V$ (the \emph{adjoint} of $f$) such that
$\langle f(v),w\rangle=\langle v,f^*(w)\rangle$ for all $v\in V$, $w\in W$. It follows from the axioms of a Hilbert $A$-module
that an adjointable map is a bounded linear operator, and that it commutes with the action of $A$. We write $\mathcal L_A(V,W)$
for the set of all adjointable maps $V\to W$. In particular, $\mathcal L_A(V,W)$ is a normed vector space, equipped with the
operator norm. It turns out that $\mathcal L_A(V)=\mathcal L_A(V,V)$ is a C*-algebra with involution $f\mapsto f^*$.

An \emph{isomorphism} of Hilbert $A$-modules $V$ and $W$ is an adjointable bijection $f\in\mathcal L_A(V,W)$, satisfying $\langle
f(v),f(v')\rangle=\langle v,v'\rangle$ for all $v,v'\in V$. Obviously, a map $f\in\mathcal L_A(V,W)$ is an isomorphism if and only
if $f^*f=\id$ and $ff^*=\id$. In particular, the \emph{automorphisms} of a Hilbert $A$-module $V$ are precisely the unitary
elements of $\mathcal L_A(V)$. More generally, we write $U(A)=\{u\in A:u^*u=uu^*=1\}$ for the set of \emph{unitaries} in an
arbitrary C*-algebra $A$.

\begin{ex}\leavevmode
  \begin{itemize}
    \item Every C*-algebra $A$ is a Hilbert $A$-module with respect to the inner product $\langle x,y\rangle=x^*y$.
    \item If $V,W$ are Hilbert $A$-modules, then also $V\oplus W$ is a Hilbert $A$-module, with inner product $\langle
      v+w,v'+w'\rangle=\langle v,v'\rangle+\langle w,w'\rangle$ for $v,v'\in V$, $w,w'\in W$.
    \item If $V$ is a Hilbert $A$-module, and $p\in\mathcal L_A(V)$ is a \emph{projection}, i.~e. $p^2=p=p^*$, then also $pV$ is a
      Hilbert $A$-module, and $V\cong pV\oplus(1-p)V$.
  \end{itemize}
\end{ex}

Now a \emph{finitely generated projective} Hilbert $A$-module is a Hilbert $A$-module which is isomorphic to $pA^k$, where
$A^k=A\oplus\cdots\oplus A$ and where $p\in\mathcal L_A(A^k)$ is a projection. Of course, finitely generated projective Hilbert
$\C$-modules are nothing but finite-dimensional complex vector spaces with Hermitian inner product.

\begin{dfn}{\cite{schick-l2-index-kk-theory-connections}}
  A Hilbert $A$-module bundle over a space $X$ is a fibre bundle $E\to X$ with typical fibre a finitely generated projective
  Hilbert $A$-module $V$, and with structure group $U(\mathcal L_A(V))$.
\end{dfn}

In particular, such a bundle may be described by local trivializations such that the transition functions take values in 
$U(\mathcal L_A(V))$. This gives a well-defined $A$-valued inner product on every fibre, such that every fibre is a Hilbert
$A$-module isomorphic to $V$. A Hilbert $\C$-module bundle is the same thing as a Hermitian vector bundle, since $U(\mathcal
L_\C(V))\cong U(n)$ is the classical group of unitary matrices.

In order to fix notations, recall that an (abstract) simplicial 
complex consists of a set $V_X$, the \emph{vertices}, and a set $X$ of non-empty finite subsets of $V_X$, the \emph{simplices},
such that every one-element set $\{p\}$ ($p\in V_X$) is contained in $X$, and such that $\emptyset\neq\rho\subset\sigma\in X$
implies that also $\rho\in X$, i.~e., $X$ is closed under taking non-empty subsets. By abuse of notation, we will refer to these 
data as the simplicial complex $X$. The \emph{dimension} of a simplex $\sigma\in X$ is $\dim(\sigma)=\#\sigma-1\in\N$. We denote
by $X_n$ the set of all simplices of dimension $n$, called the \emph{$n$-simplices}.

If $X$ is a simplicial complex and $k\geq 0$ is a number, then $X^{(k)}$ is the simplicial complex having the same set of vertices
as $X$, and the simplices of $X^{(k)}$ are precisely the simplices of $X$ which have dimension at most $k$.

The \emph{geometric realization} of a simplicial complex $X$ is the topological space whose underlying set $|X|$ is the set of all
real linear combinations $\sum_{p\in V_X}\lambda_p\cdot p$, such that
\begin{itemize}
  \item the set of those $p\in V_X$ with $\lambda_p\neq 0$ is a simplex of $X$ (and in particular, there are only finitely many
    non-zero $\lambda_p$), and
  \item $\sum_{p\in V_X}\lambda_p=1$.
\end{itemize}
For every simplex $\sigma\in X_n$, after a choice of ordering $\sigma=\{p_0,\ldots,p_n\}$ of its vertices, there is an injective
map 
\[
  j_\sigma\colon\Delta^n\to|X|,\quad (\lambda_0,\ldots,\lambda_n)\mapsto\sum_{i=0}^n\lambda_i\cdot p_i.
\]
Here the \emph{standard $n$-simplex} $\Delta^n\subset\R^{n+1}$ is the convex hull of the standard unit vectors in $\R^{n+1}$,
i.~e. the set of all tuples $(\lambda_0,\ldots,\lambda_n)$ such that $\sum_{i=0}^n\lambda_i=1$. Now $|X|$ is equipped with the
weakest topology such that all $j_\sigma$ are continuous. This means that a set $U\subset|X|$ is open if and only if all
$j_\sigma^{-1}U$ are open. In particular, the maps $j_\sigma$ are embeddings of topological subspaces. We denote by
$|\sigma|=j_\sigma(\Delta^n)$ the geometric realization of the simplex $\sigma\in X_n$. Thus, the elements of $|\sigma|$ are
convex combinations of the vertices of $\sigma$.

Now let $X$ be a simplicial complex, and let $E\to|X|$ be a Hilbert $A$-module bundle modelled on the finitely generated 
projective Hilbert $A$-module $V$, for instance $V=\C^n$. Suppose that for each 
simplex $\sigma\in X_n$, we have a trivialization $\Phi_\sigma\colon|\sigma|\times V\to j_\sigma^*E$, i.~e. 
$\Phi_\sigma|_{\{x\}\times V}$ is an isomorphism of Hilbert $A$-modules for each $x\in\Delta^n$. For
ordinary Hermitian bundles, this simply means that $\Phi_\sigma$ respects the inner product in every fibre.

Now consider a simplex $\sigma\in X_n$ and some sub-simplex $\rho\subset\sigma\in X_k$. We define the transition function
\[
  \Psi_{\rho\subset\sigma}\colon|\rho|\to U(\mathcal L_A(V)),\quad x\mapsto\Phi_\rho(x,\cdot)^{-1}\circ\Phi_\sigma(x,\cdot).
\]

\begin{dfn}
  An \emph{$\epsilon$-flat family of trivializations} (where $\epsilon>0$ is a number) of a Hilbert $A$-module bundle $E\to|X|$ 
  consists of trivializations $\Phi_\sigma\colon|\sigma|\times V\to j_\sigma^*E$, such that the transition functions 
  $\Psi_{\rho\subset\sigma}\colon|\rho|\to U(\mathcal L_A(V))$ are Lipschitz functions with Lipschitz constant at most $\epsilon$. 
  Here $|\rho|$ carries the metric such that $j_\rho\colon\Delta^k\to|\rho|$ is an isometry. An 
  \emph{$\epsilon$-flat bundle} is a Hilbert $A$-module bundle together with an $\epsilon$-flat family of trivializations. An 
  \emph{almost flat bundle} is an $\epsilon$-flat bundle for some $\epsilon$.
\end{dfn}

Naturally, the equivalence class of $E$ is uniquely determined by the transition functions. Therefore, an equivalent formulation
of an $\epsilon$-flat bundle could simply specify a family of $\epsilon$-Lipschitz transition functions satisfying appropriate
cocycle conditions.

\subsection{Example: Hilbert module bundles with connections}

An important class of examples for $\epsilon$-flat bundles comes from Riemannian geometry. Namely, let $E\to M$ be a \emph{smooth}
Hilbert $A$-module bundle over a Riemannian manifold $M$. This means that $M$ can be covered by open sets $U_i$ such that $E$ can
be trivialized over each $U_i$, and such that the transition functions $U_i\cap U_j\to U(\mathcal L_A(V))$ are smooth. A
\emph{connection} on $E$ \cite{schick-l2-index-kk-theory-connections} is a linear map 
$\nabla\colon\mathcal C^\infty(TM)\otimes\mathcal C^\infty(E)\to\mathcal C^\infty(E)$, $X\otimes s\mapsto\nabla_Xs$, such that 
\begin{itemize}
  \item $\nabla_X(s\cdot f)=s\cdot(Xf)+(\nabla_Xs)\cdot f$, and
  \item $\nabla_{gX}s=g\nabla_Xs$
\end{itemize}
for every $X\in\mathcal C^\infty(TM)$, $f\in\mathcal C^\infty(M;A)$, $g\in\mathcal C^\infty(M)$ and $s\in\mathcal C^\infty(E)$.
Such a connection is called \emph{compatible} (with the metric) if $X\langle s,s'\rangle=\langle\nabla_Xs,s'\rangle+\langle
s,\nabla_Xs'\rangle$ for all $X\in\mathcal C^\infty(TM)$, $s,s'\in\mathcal C^\infty(E)$.

\begin{ex}
  If $E=M\times V$ is trivial, then a compatible connection $\partial$ on $E$ is given by partial derivative 
  $\partial_Xs=\left.\frac d{dt}\right|_{t=0}s(\gamma(t))$ where $\gamma$ is a curve with $\gamma'(0)=X$.
\end{ex}

For any C*-algebra $A$, we consider the subspace of \emph{skew-adjoint} elements $\mathfrak u(A)=\{x\in A: x^*+x=0\}$.

\begin{prop}
  Every compatible connection $\nabla$ on a trivial bundle $E=M\times V$ is of the form $\nabla_Xs=\partial_Xs+\Gamma(X)s$ where 
  $\Gamma\in\mathcal C^\infty(T^*M\otimes\mathfrak u(\mathcal L_A(V)))$ is a smooth $\mathfrak u(\mathcal L_A(V))$-valued 1-form.
\end{prop}
\begin{proof}
  It follows right from the definition of a connection that the map $\Gamma$, defined by $\Gamma(X)s:=\nabla_Xs-\partial_Xs$,
  is tensorial in the sense that $\Gamma(X)(s\cdot f)=\Gamma(X)s\cdot f$ for any $f\in\mathcal C^\infty(M)$. If $V=A^n$ is free,
  as in the vector space case it may be shown that this is exactly the condition for $\Gamma$ to define a 
  $\mathcal L_A(V)$-valued 1-form. If $V$ is not free, choose a finitely generated projective Hilbert $A$-module $V'$ such that
  $V\oplus V'\cong A^n$ for some $n$. Set $E'=M\times V'$, so that $E\oplus E'\cong M\times A^n$. For sections 
  $s\in\mathcal C^\infty(E),s'\in\mathcal C^\infty(E'),X\in\mathcal C^\infty(TM)$, let $\tilde\Gamma(X)(s+s')=\Gamma(X)s$. This is
  obviously still tensorial, so by the discussion above $\Gamma(X)s|_x=\tilde\Gamma(X)(s+0)|_x$ depends only on $s(x)$ for every
  $x\in M$. Using that $\nabla$ an $\partial$ are compatible, it is easy to show that 
  $\Gamma(X)_x\in\mathfrak u(\mathcal L_A(V))$.
\end{proof}

Now the \emph{curvature} induced by $\nabla$ is defined by the formula
\[
  \mathcal R^\nabla(X,Y)s=\nabla_X\nabla_Ys-\nabla_Y\nabla_Xs-\nabla_{[X,Y]}s
\]
for $X,Y\in\mathcal C^\infty(TM)$, $s\in\mathcal C^\infty(E)$. As in the case of vector bundles $E$, one immediately sees that
$\mathcal R^\nabla$ is tensorial in all three entries, i.~e. $\mathcal R^\nabla(gX,Y)=g\mathcal R^\nabla(X,Y)=\mathcal
R^\nabla(X,gY)$ and $\mathcal R^\nabla(X,Y)(s\cdot f)=(\mathcal R^\nabla(X,Y)s)\cdot f$ for $X,Y\in\mathcal C^\infty(TM)$,
$s\in\mathcal C^\infty(E)$, $f\in\mathcal C^\infty(M;A)$ and $g\in\mathcal C^\infty(M)$. This implies that $\mathcal R^\nabla$
defines a $\mathcal L_A(E)$-valued 2-form.

Given a smooth Hilbert module bundle $E\to M$ and a smooth map $f\colon N\to M$, the bundle $f^*E$ is obviously also a smooth
Hilbert module bundle over $N$. We denote the canonical bundle map by $\hat f\colon f^*E\to E$. Of course, every section
$s\in\mathcal C^\infty(E)$ induces a section $f^*s\in\mathcal C^\infty(f^*E)$ which is determined uniquely by the property that
$\hat f\circ f^*s=s\circ f$. Now if $E$ is equipped with a connection $\nabla$ then a connection $f^*\nabla$ on $f^*E$ is 
determined uniquely by the property that $(f^*\nabla)_X(f^*s)=f^*(\nabla_{f_*X}s)$.

Let $\gamma\colon(a,b)\to M$ be a smooth curve. 
A section $s$ of $E$ along $\gamma$ is a section of the pullback bundle $\gamma^*E$, and such a section is called \emph{parallel}
if $(\gamma^*\nabla)_{\partial_t}s=0$ on $(a,b)$. As for ordinary vector bundles, for every $t\in(a,b)$ and every
$s(t)\in(\gamma^*E)_t$, there exists a parallel section $s$ along $\gamma$ which coincides with $s(t)$ at point $t$.
Thus, if $\gamma\colon[0,1]\to M$ is a smooth curve connecting $p=\gamma(0)$ and $q=\gamma(1)$, we may 
define a parallel transport map
$T_\gamma\colon E_p\to E_q$ by mapping an element $v\in E_p$ to $\hat\gamma(s(1))$ where $s\colon[0,1]\to\gamma^*E$ is the unique
parallel section along $\gamma$ which satisfies $\hat\gamma(s(0))=v$.

Now if $M$ is triangulated (in this paper, a triangulated manifold will always be smoothly triangulated in the sense that the
simplices are smoothly embedded) and if $|\sigma|\subset M$ is the embedding of a simplex, then we may trivialize $E|_{|\sigma|}$ 
over $|\sigma|$ by choosing an isomorphism of $V$ with the fibre over the barycenter $b_\sigma\in|\sigma|$ and composing this 
isomorphism with the parallel transport outwards along curves of the form $t\mapsto tx+(1-t)b_\sigma$ for $x\in|\partial\sigma|$. 
Here the barycenter $b_\sigma\in|\sigma|$ is the point $b_\sigma=\sum_{v\in\sigma}\frac 1{\#\sigma}\cdot v$. This 
procedure then gives a trivialization $\Phi_\sigma\colon\Delta^n\times V\to j_\sigma^*E$, and it is easy to see that this 
trivialization via parallel transport preserves the inner product if $\nabla$ is compatible and if the isomorphism at $b$
preserves the inner product. Thus, in this case the transition functions take their values in $U(\mathcal L_A(V))$.

\begin{thm}\label{thm:bundles-with-connection-are-epsilon-small}
  Let $M$ be a triangulated Riemannian manifold. Then there is a constant $c(M)>0$ such that the following holds. Let $E\to M$ be 
  a Hilbert $A$-module bundle over an arbitrary C*-algebra $A$. Assume that $E$ is equipped with a compatible connection, and let
  $\Phi_\sigma$ be trivializations via parallel transport as defined above. Assume that
  \[
    \|\mathcal R^\nabla\|_g=\sup_{\substack{X\wedge Y\in\Lambda^2TM\\\|X\wedge Y\|_g\leq 1}}\|\mathcal
    R^\nabla(X,Y)\|_{\text{op}}\leq\epsilon.
  \]
  Then the $\Phi_\sigma$ constitute an $c(M)\epsilon$-flat family of trivializations on $E$.
\end{thm}

For the proof of theorem \ref{thm:bundles-with-connection-are-epsilon-small}, we will need that parallel transport along loops
which bound a small area is close to the identity. The proof of this statement is not so easily found in the literature, in 
particular not for Hilbert module bundles, so I will give it in appendix \ref{app:parallel-transport-curvature}. The precise
formulation is the following.

\begin{prop}\label{prop:transport-along-small-curves-is-small}
  Let $f\colon[0,1]\times[0,1]\to M$ be a smooth map. We denote parallel transport in $E$ along the curve 
  $f(\partial([0,1]\times[0,1]))$ by the symbol $P_{\partial f}$. Further, we consider a Hilbert $A$-module
  bundle $E\to M$ with compatible connection $\nabla$, and the associated curvature tensor $\mathcal R^\nabla$. Then
  \[
    \|P_{\partial f}-\id\|\leq\int_0^1\int_0^1\|\mathcal R^\nabla(\partial_sf(s,t)\wedge\partial_tf(s,t))\|\,ds\,dt.
  \]
\end{prop}

Using this, we can prove theorem \ref{thm:bundles-with-connection-are-epsilon-small}.

\begin{proof}[Proof of theorem \ref{thm:bundles-with-connection-are-epsilon-small}]
  Let $\rho\subset\sigma$ be simplices of $M$. We want to show that the transition function
  $\Psi_{\rho\subset\sigma}\colon|\rho|\to U(\mathcal L_A(V))$ is Lipschitz with Lipschitz constant bounded by a multiple of
  $\epsilon$. Since multiplication with a constant unitary does not change the Lipschitz constant of a map, we may assume that 
  $\Phi_\rho(b_\rho,v)$ is given by parallel transport of $\Phi_\sigma(b_\sigma,v)$ along the curve 
  $t\mapsto tb_\rho+(1-t)b_\sigma$.

  For every pair of points $a,b\in|\sigma|$ we denote by $T_{a,b}\colon E_a\to E_b$ the parallel transport map along the straight
  line segment $t\mapsto tb+(1-t)a$. Note that $T_{a,b}$ preserves the inner product since the connection is compatible, and that
  $T_{a,b}^{-1}=T_{b,a}$. Then we have that
  \begin{align*}
    \Phi_\sigma(p,v)&=T_{b_\sigma,p}\Phi_\sigma(b_\sigma,v),\\
    \Phi_\rho(p,v)&=T_{b_\rho,p}T_{b_\sigma,b_\rho}\Phi_\sigma(b_\sigma,v)
  \end{align*}
  for every point $p\in|\rho|$. This immediately implies that
  \[
    \Psi_{\rho\subset\sigma}(p)=\Phi_\rho(p,\cdot)^{-1}\Phi_\sigma(p,\cdot)
    =\Phi_\sigma(b_\sigma,\cdot)^{-1}T_{b_\rho,b_\sigma}T_{p,b_\rho}T_{b_\sigma,p}\Phi_\sigma(b_\sigma,\cdot).
  \]

  Now consider arbitrary points $x,y\in|\rho|$. The above equations imply that 
  \begin{align*}
    \|\Psi_{\rho\subset\sigma}(x)-\Psi_{\rho\subset\sigma}(y)\|
    &=\|\Psi_{\rho\subset\sigma}(x)\Psi_{\rho\subset\sigma}(y)^{-1}-\id\|\\
    &=\|\Phi_\sigma(b_\sigma,\cdot)^{-1}T_{b_\rho,b_\sigma}T_{x,b_\rho}T_{b_\sigma,x}T_{y,b_\sigma}T_{b_\rho,y}T_{b_\sigma,b_\rho}
    \Phi_\sigma(b_\sigma,\cdot)-\id\|\\
    &=\|T_{x,b_\rho}T_{b_\sigma,x}T_{y,b_\sigma}T_{b_\rho,y}-\id\|\\
    &=\|T_{b_\rho,y}T_{x,b_\rho}T_{b_\sigma,x}T_{y,b_\sigma}-\id\|\\
    &=\|(T_{b_\rho,y}T_{x,b_\rho}T_{y,x})(T_{x,y}T_{b_\sigma,x}T_{y,b_\sigma})-\id\|\\
    &=\|T_{b_\rho,y}T_{x,b_\rho}T_{y,x}-T_{b_\sigma,y}T_{x,b_\sigma}T_{y,x}\|\\
    &\leq\|T_{b_\rho,y}T_{x,b_\rho}T_{y,x}-\id\|+\|T_{b_\sigma,y}T_{x,b_\sigma}T_{y,x}-\id\|
  \end{align*}
  since $T_{b_\rho,b_\sigma}$ and $\Phi_\sigma(b_\sigma,\cdot)$ preserve the norm. Thus, we have to show that transport along
  triangles of the form $\Delta(y,x,b_\rho)$ is close to the identity whenever $x$ and $y$ are close. This is true
  because those triangles obviously bound disks of a small area, so we may use proposition
  \ref{prop:transport-along-small-curves-is-small} to obtain the result. To make this rigorous, one only has to note that, by an
  easy compactness argument, we may assume that the metric on the simplex $|\sigma|$ equals the standard metric. Then, the
  triangle bounds an area of at most $\frac 12d(x,y)\cdot\diam(\Delta^n)=\frac 12\sqrt 2\cdot d(x,y)$.
\end{proof}

\section{The trivialization lemma}

\label{sec:trivialization-lemma}

The goal of this section is to prove the trivialization lemma which states that an $\epsilon$-flat bundle over a simply connected
space is trivial if $\epsilon$ is small enough. This is the basic result which enables us to extend $\epsilon$-flat bundles to
larger $\Delta$-sets under certain conditions. Namely, one can apply the trivialization lemma to the sphere $S^{n-1}$ if $n>2$ to
trivialize the bundle over the boundary $\partial\Delta^n$ of a simplex and thus to extend the bundle over the whole simplex
$\Delta^n$. 

\subsection{Transport in the 1-skeleton}

We first want to show that for $\epsilon$-flat bundles, transport along contractible simplicial loops is close to the identity, in
analogy with proposition \ref{prop:transport-along-small-curves-is-small}. This will be one of the main ingredients in the proof 
of the trivialization lemma. 

For any simplicial complex $X$, we define the \emph{simplicial path category} $\mathcal P_X$ as follows: Objects of $\mathcal P_X$
are the vertices of $X$, and morphisms from $v_0$ to $v_k$ are \emph{simplicial paths}, i.~e. tuples $(v_0,\ldots,v_k)$ such that 
$\{v_i,v_{i+1}\}\in X_1$. One should imagine simplicial paths as concatenations of the paths $t\mapsto (1-t)v_i+tv_{i+1}$. The
composition of two simplicial paths $\Gamma=(v_0,\ldots,v_k)$ and $\Gamma'=(v_k,\ldots,v_{k+l})$ is to be the path
$\Gamma*\Gamma'=(v_0,\ldots,v_{k+l})$.

Let $E\to|X|$ be an $\epsilon$-flat bundle, and let $\sigma=\{v_0,v_1\}\in X_1$ be an edge. Then \emph{transport along 
$(v_0,v_1)$} is the isomorphism of Hilbert $A$-modules
\[
  T_{(v_0,v_1)}=\Phi_\sigma(v_1,\cdot)\Phi_\sigma(v_0,\cdot)^{-1}\colon E_{v_0}\to E_{v_1}.
\]
If $\Gamma=(v_0,\ldots,v_k)$ is a simplicial path then \emph{transport along $\Gamma$} is defined as
\[
  T_\Gamma=T_{(v_{k-1},v_k)}\cdots T_{(v_0,v_1)}\colon E_{v_0}\to E_{v_k}.
\]
Of course, if $\Gamma=(\Gamma_1,\Gamma_2)$ then $T_\Gamma=T_{\Gamma_2}T_{\Gamma_1}$, so the associations $v\mapsto E_v$,
$\Gamma\mapsto T_\Gamma$ define a functor $\mathcal P_X\to\mathcal M_A$ into the category of Hilbert $A$-modules and Hilbert
$A$-module isomorphisms. 

We want to analyse the transport along contractible simplicial loops $\Gamma$. We first consider the special case that 
$\Gamma$ is the boundary curve of a 2-simplex.

\begin{prop}\label{prop:transport-entlang-2-simplizes}
  Let $\Gamma=(v_0,v_1,v_2,v_0)$ be the simplicial loop along the boundary of a 2-simplex $\sigma=\{v_0,v_1,v_2\}\in X_2$. If $E$ 
  is an $\epsilon$-flat bundle over $X$ with $\epsilon\leq 1/\sqrt 2$, then
  \[
    \|T_\Gamma-\id\|\leq\epsilon\cdot 7\sqrt 2.
  \]
\end{prop}
\begin{proof}
  By definition of the transition function we have that
  \[
    \Phi_\rho(x,\cdot)=\Phi_\sigma(x,\cdot)\circ\Psi_{\rho\subset\sigma}(x,\cdot)^{-1}
  \]
  for $\rho=\{v_i,v_j\}\subset\sigma$. It follows that 
  \[
    T_{(v_i,v_j)}=\Phi_\sigma(v_j,\cdot)\Psi_{\rho\subset\sigma}(v_j)^{-1}\Psi_{\rho\subset\sigma}(v_i)
    \Phi_\sigma(v_i,\cdot)^{-1}.
  \]

  By definition we have that $T_\Gamma=T_{(v_2,v_0)}T_{(v_1,v_2)}T_{(v_0,v_1)}$. Together it follows that 
  \begin{align*}
    T_\Gamma=\Phi_\sigma(v_0,\cdot)&\cdot(\Psi_{\{v_2,v_0\}\subset\sigma}(v_0)^{-1}
    \Psi_{\{v_2,v_0\}\subset\sigma}(v_2))\cdot\phantom{x}\\
    &\cdot(\Psi_{\{v_1,v_2\}\subset\sigma}(v_2)^{-1}\Psi_{\{v_1,v_2\}\subset\sigma}(v_1))\cdot\phantom{x}\\
    &\cdot(\Psi_{\{v_0,v_1\}\subset\sigma}(v_1)^{-1}\Psi_{\{v_0,v_1\}\subset\sigma}(v_0))\cdot
    \Phi_\sigma(v_0,\cdot)^{-1}
  \end{align*}
  holds. Since the $\Psi_{\rho\subset\sigma}$ are $\epsilon$-Lipschitz maps and the vertices have distance $\sqrt 2$ we conclude
  that
  \[
    \|\Psi_{\rho\subset\sigma}(v_j)^{-1}\Psi_{\rho\subset\sigma}(v_i)-\id\|\leq\epsilon\sqrt 2.
  \]
  Now we apply the following very useful lemma.

  \begin{lem}\label{lem:very-useful-lemma}
    Let $0<\epsilon\leq 1$. Let $V_1,\ldots,V_n$ be Hilbert $A$-modules. Let furthermore
    \[
      P=A_nB_nA_{n-1}B_{n-1}\cdots A_1B_1\colon V_1\to V_1
    \]
    where all $A_i\colon V_i\to V_{i+1}$ are isomorphisms of Hilbert $A$-modules, $V_{n+1}=V_1$, $A_n\cdots A_1=\id$, and where
    all the $B_i\colon V_i\to V_i$ are linear with $\|B_i-\id\|<\epsilon$. Then also $\|P-\id\|<c\epsilon$, and $c=c(n)=2^n-1$
    depends only on $n$.
  \end{lem}
  \begin{proof}
    Replace $B_i$ by $(B_i-\id)+\id$ in the definition of $P$ and expand. Then one sees that $P-\id$ is the sum of $2^n-1$ linear
    maps having norm bounded by $\epsilon$, since $\epsilon\leq 1$. The claim follows using the triangle inequality.
  \end{proof}

  To finish the proof of proposition \ref{prop:transport-entlang-2-simplizes} we can use the lemma with
  $n=3$. Thus we get $\|T_\Gamma-\id\|\leq c\epsilon\sqrt 2$ with $c=2^3-1=7$.
\end{proof}

Of course, the constants used in lemma \ref{lem:very-useful-lemma} and proposition \ref{prop:transport-entlang-2-simplizes} are
somewhat arbitrary: If we require other bounds for $\epsilon$, then we get other constants. However, since we will only
consider $\epsilon$-flat bundles when $\epsilon$ is small, this extra flexibility is not important here.

We recall the following well-known description of the fundamental group of $X$. The \emph{homotopy simplicial path category} 
$\mathcal P_X'$ is the quotient category of $\mathcal P_X$ modulo the congruence relation 
generated by identifying $(v_0,v_1,v_0)$ and $(v_0)$ for every $\{v_0,v_1\}\in X_1$, and by identifying $(v_0,v_1,v_2,v_0)$ and 
$(v_0)$ for every $\{v_0,v_1,v_2\}\in X_2$. 

The \emph{path groupoid} of a space $X$ is the category $\Pi_X$ which has as objects the points in $X$ and as morphisms homotopy
classes of paths $[0,1]\to X$ relative to the endpoints, where composition is given as concatenation. In particular,
$\pi_1(X,x_0)=\Pi_X(x_0,x_0)$. 

\begin{prop}\label{prop:presentation-of-fundamental-group}
  If $X$ is a simplicial complex, the natural functor $\mathcal P_X\to\Pi_X$ descends to a fully faithful functor 
  $\mathcal P_X'\to\Pi_X$. In particular, $\pi_1(X,x_0)\cong\mathcal P_X'(x_0,x_0)$ for every vertex $x_0$.\qed
\end{prop}

This implies that if $\Gamma\in\mathcal P_X(x_0,x_0)$ is a contractible simplicial loop then $\Gamma$ arises by a finite sequence
of omissions or insertions of pieces of the form $(v_0,v_1,v_0)$ where $\{v_0,v_1\}\in X_1$, or of pieces of the form
$(v_0,v_1,v_2,v_0)$ where $\{v_0,v_1,v_2\}\in X_2$. Note that omissions or insertions of the first type do not change the
transport map. The \emph{homotopical complexity} $\hc(\Gamma)$ of such a contractible simplicial loop $\Gamma$ is the minimum 
number of insertions or omissions of the second form which is needed to obtain $\Gamma$ from the empty loop.

\begin{thm}\label{thm:transport-along-contractible-loops}
  Let $\Gamma$ be a contractible simplicial loop in $X$ and $n=\hc(\Gamma)$ its homotopical complexity as defined above. Then
  there are constants $c(n),\delta(n)>0$, depending only on $n$, such that for every $\epsilon$-flat bundle over $X$ with
  $\epsilon\leq\delta(n)$ the transport along $\Gamma$ satisfies the inequality
  \[
    \|T_\Gamma-\id\|\leq c(n)\epsilon.
  \]
\end{thm}
\begin{proof}
  We prove the claim by induction over $n$. If $n=0$ there is nothing to do, since insertions or omissions of the first type
  do not alterate the transport map. Thus we assume that $\Gamma=(\Gamma_1,\hat\Gamma,\Gamma_2)$ with
  $\hat\Gamma=(v_0,v_1,v_2,v_0)$ and $\hc(\Gamma_1,\Gamma_2)=n-1$. Since the transport maps are isomorphisms of Hilbert 
  $A$-modules, we get that
  \[
    \|T_\Gamma-\id\|=\|T_{\Gamma_2}T_{\hat\Gamma}T_{\Gamma_1}-\id\|=\|T_{\hat\Gamma}T_{(\Gamma_1,\Gamma_2)}-\id\|.
  \]
  By proposition \ref{prop:transport-entlang-2-simplizes}, $\|T_{\hat\Gamma}-\id\|\leq c(1)\epsilon$ where $c(1)=7\sqrt 2$, and by
  induction we may assume that $\|T_{(\Gamma_1,\Gamma_2)}-\id\|\leq c(n-1)\epsilon$ if $\epsilon\leq\min\{1/\sqrt
  2,\delta(n-1)\}$. If we now let $\delta(n)=\min\{c(1)^{-1},c(n-1)^{-1},\delta(n-1)\}$ then $\max\{c(1),c(n-1)\}\epsilon\leq 1$,
  so we may use lemma \ref{lem:very-useful-lemma} to show that
  \[
    \|T_\Gamma-\id\|\leq 3\max\{c(1),c(n-1)\}\epsilon.
  \]
  If, on the other hand, $\Gamma=(\Gamma_1,\Gamma_2)$ and $\hc(\Gamma_1,\hat\Gamma,\Gamma_2)\leq n-1$, then transport along
  $\Gamma$ is the same thing as transport along the curve $(\Gamma_1,v_0,v_2,v_1,v_0,\hat\Gamma,\Gamma_2)$ and we may use
  the same proof to show that $\|T_\Gamma-\id\|\leq 3\max\{c(1),c(n-1)\}\epsilon$. The claim of the theorem follows with
  $c(n)=3\max\{c(1),c(n-1)\}$.
\end{proof}

\subsection{The trivialization lemma}

Now we can use the previous results to prove the trivialization lemma. This states that $\epsilon$-flat bundles over simply
connected spaces are trivial if $\epsilon$ is small enough. We can further achieve that the transition functions from the
$\epsilon$-flat family to the global trivialization are Lipschitz with small Lipschitz constant.

This is to say that if $E\to|X|$ is an almost flat Hilbert $A$-module bundle with global trivialization 
$\Phi_X\colon|X|\times V\to E$, then, as before, we obtain the transition function
\[
  \Psi_{\rho\in X}\colon|\rho|\to U(\mathcal L_A(V)),\quad x\mapsto\Phi_\rho(x,\cdot)^{-1}\circ\Phi_X(x,\cdot).
\]
Now $\Phi_X$ is a \emph{global $\epsilon$-trivialization} if the maps $\Psi_{\rho\in X}$ are $\epsilon$-Lipschitz for 
every simplex $\rho\in X$.

The proof of the trivialization lemma relies heavily on the following extension property for maps into the unitaries of a
C*-algebra, for which we give a proof in appendix \ref{app:unitary-elements-of-cstar-algebras}.

\begin{lem}\label{lem:extension-of-unitary-valued-lipschitz-maps}
  There is a number $C>0$ with the following property. Let $A$ be a C*-algebra with unit, and denote $U(A)$ the set of unitary
  elements of $A$. Let further $\alpha_0\colon S^{n-1}\to U(A)$ be a $\lambda$-Lipschitz map. If $\diam\alpha_0(S^{n-1})\leq\frac
  12$, there exists an extension $\alpha\colon D^n\to U(A)$ with Lipschitz constant at most $C\lambda$. Here, the constant $C$ is
  independent of $A$.
\end{lem}

For the statement of the trivialization lemma, recall that a \emph{tree} in a simplicial complex $X$ is a contractible subcomplex
of the $1$-skeleton of $X$. Every simplicial complex contains a maximal tree, and it is a basic fact that a tree is maximal in $X$
if and only if it contains all the vertices of $X$.

\begin{thm}[Trivialization lemma]\label{thm:trivialization lemma}
  Let $X$ be a simplicial complex, and let $T\subset X$ be a maximal tree in $X$. Then there are constants $C(X),\delta(X)>0$ such 
  that the following holds: Let $E\to X$ be an $\epsilon$-flat bundle where $\epsilon\leq\delta(X)$. Suppose further that for 
  every simplex $\sigma=\{x,y\}\in X_1$, there is a simplicial loop $\Gamma_\sigma=(x,y)*\Gamma_\sigma^0$ such that 
  $\Gamma_\sigma^0$ is a simplicial path in $T$ and such that $\|T_{\Gamma_\sigma}-\id\|\leq\epsilon$. Then $E$ admits a global 
  $C(X)\epsilon$-trivialization.
\end{thm}
\begin{proof}
  Firstly, we want to prove that such a trivialization exists in the case that $X$ is a graph, i.~e. $1$-dimensional. First note
  that it is trivial to construct a global $0$-trivialization over the tree $T$. Now let $\sigma=\{x,y\}\in X_1$ be a simplex, and 
  let $\Gamma_\sigma=(x,y)*\Gamma_\sigma^0$ be a simplicial loop as in the assumption of the theorem. Consider the transition 
  function
  \[
    \Psi_{\sigma,T}\colon\{x,y\}\to U(\mathcal L_A(V)),\quad x\mapsto\Phi_\sigma(x,\cdot)^{-1}\circ\Phi_T(x,\cdot).
  \]
  Then we have that
  \begin{align*}
    \|\Psi_{\sigma,T}(x)-\Psi_{\sigma,T}(y)\|&=\|\Phi_T(y,\cdot)^{-1}T_{(x,y)}\Phi_T(x,\cdot)-\id\|\\
                                             &=\|\Phi_T(y,\cdot)^{-1}T_{\Gamma_\sigma^0}^{-1}\Phi_T(y,\cdot)
                                              \Phi_T(y,\cdot)^{-1}T_{\Gamma_\sigma}\Phi_T(x,\cdot)-\id\|\\
                                             &=\|T_{\Gamma_\sigma}-\id\|\leq\epsilon
  \end{align*}
  by assumption and using that $\Phi_T(y,\cdot)^{-1}T_{\Gamma_\sigma^0}^{-1}\Phi_T(y,\cdot)=\id$ because $\Phi_T$ is a
  $0$-trivialization. Now by lemma
  \ref{lem:extension-of-unitary-valued-lipschitz-maps}, the map $\Psi_{\sigma,T}$ posseses a $C\epsilon$-Lipschitz extension
  $\Psi_{\sigma,T}'$. Of course, we have that $\Phi_T(p,\cdot)=\Phi_\sigma(p,\cdot)\circ\Psi_{\sigma,T}'(p)$ for $p\in\{x,y\}$, so
  we can extend $\Phi_T$ by the same formula onto $|\sigma|$, and the transition function for $\sigma$ will be
  $C\epsilon$-Lipschitz.

  In the general case, we proceed by induction on the dimension of $Y$. Thus, we may assume that we already have a global
  $\epsilon$-trivialization over the $k$-skeleton of $Y$, where $k\geq 1$, and we want to extend it to the $(k+1)$-skeleton. Let
  $S=Y^{(k)}=\bigcup_{i\leq k}Y_i$, and let $\Phi_S$ be the corresponding global trivialization. Given a $(k+1)$-simplex $\rho$, 
  the map 
  \[
    \Psi_{\rho,S}\colon|\partial\rho|\to U(\mathcal L_A(V)),\quad x\mapsto\Phi_\sigma(x,\cdot)^{-1}\circ\Phi_S(x,\cdot)
  \]
  is $\epsilon$-Lipschitz on every simplex of $|\partial\rho|$. By connectedness of $|\partial\rho|$, it is globally 
  $\epsilon$-Lipschitz, and we may use lemma \ref{lem:extension-of-unitary-valued-lipschitz-maps} again to find a Lipschitz 
  extension on the whole of $|\rho|$. As above, this completes the proof in the general case.
\end{proof}

\begin{korl}\label{korl:trivialization-lemma-simply-connected}
  Let $X$ be a finite contractible simplicial complex. Then there are constants $C(X),\delta(X)>0$ such that the following holds:
  Let $E\to X$ be an $\epsilon$-flat bundle where $\epsilon\leq\delta(X)$. Then $E$ admits a global $C(X)\epsilon$-trivialization.
\end{korl}
\begin{proof}
  Choose a tree $T\subset X$, and arbitrary paths $\Gamma_\sigma=(x,y)*\Gamma_\sigma^0$ for every simplex $\sigma=\{x,y\}\in
  X_1-T$, such that $\Gamma_\sigma^0$ is a simplicial path in $T$. Then these curves satisfy the assumptions from theorem
  \ref{thm:trivialization lemma} because of theorem \ref{thm:transport-along-contractible-loops}.
\end{proof}

\section{Applications of the trivialization lemma}

\label{sec:first-applications}

\subsection{Subdivision}

As a first application of the trivialization lemma, we show that almost flat bundles are invariant under barycentric subdivision.

The \emph{barycentric subdivision} of a simplicial complex $X$ is the simplicial complex $\mathcal S(X)$ whose vertices are
$V_{\mathcal S(X)}=X$ and whose simplices are $\mathcal S(X)=\{\{\sigma_0,\ldots,\sigma_k\}:\sigma_i\subset\sigma_{i+1}\}$.
This deserves the name subdivision, as the following shows.

\begin{lem}[{\cite{spanier-algebraic-topology}}]
  For any simplicial complex $X$, there is a natural homeomorphism $\Xi_X\colon|\mathcal S(X)|\to|X|$ which is given on vertices 
  by $\Xi_X(\sigma)=\sum_{v\in\sigma}\frac 1{\#\sigma}v$, and which is affine linear on every simplex of $\mathcal S(X)$.\qed
\end{lem}

Thus, we may identify $|\mathcal S(X)|$ and $|X|$, and in particular, a bundle $E\to|X|$ is the same thing as a bundle
$E\to|\mathcal S(X)|$. Now suppose that we have an $\epsilon$-flat bundle $E\to|X|$ (with respect to the triangulation $X$). If
$\rho$ is a simplex of $\mathcal S(X)$ then the image of $|\rho|$ in $|X|$ is contained in the realization of a simplex 
$\bigcup\rho\in X$, and the embedding $\Xi_X\colon|\rho|\to|\bigcup\rho|$ induces a trivialization
\[
  \Phi_\rho=\Phi_{\bigcup\rho}\circ(\Xi_X\times\id)\colon|\rho|\times V\to E|_{|\rho|}.
\]
Now if $\rho'\subset\rho$, then $\Psi_{\rho'\subset\rho}=\Psi_{\bigcup\rho'\subset\bigcup\rho}\circ\Xi_X$.

This shows that almost flat bundles over $X$ induce almost flat bundles over $\mathcal S(X)$. The opposite statement is also true,
as we will see momentarily. For the proof of this statement, we will need the following useful observation.

\begin{lem}\label{lem:product-of-lipschitz-maps}
  Let $X$ be a metric space and let $f,g\colon X\to U(\mathcal L_A(V))$ be $\epsilon$-Lipschitz maps for some $\epsilon>0$. Then
  the map $X\to U(\mathcal L_A(V))$, $x\mapsto f(x)\circ g(x)$, is $3\epsilon$-Lipschitz.
\end{lem}
\begin{proof}
  If $x,y\in X$, then
  \[
    \|f(x)g(x)-f(y)g(y)\|=\|f(y)^{-1}f(x)g(x)g(y)^{-1}-\id\|\leq 3\epsilon
  \]
  by lemma \ref{lem:very-useful-lemma}.
\end{proof}

\begin{prop}\label{prop:subdivision}
  Let $X$ be a finite-dimensional simplicial complex, and consider its barycentric subdivision $\mathcal S(X)$. Then there are
  constants $C_1,C_2,\delta>0$, depending only on the dimension of $X$, such that:
  \begin{itemize}
    \item every $\epsilon$-flat bundle over $X$ is canonically a $C_1\epsilon$-flat bundle over $\mathcal S(X)$, and
    \item if $\epsilon\leq\delta$, then every $\epsilon$-flat bundle over $\mathcal S(X)$ admits a $C_2\epsilon$-flat family of
      trivializations with respect to the triangulation $X$.
  \end{itemize}
\end{prop}
\begin{proof}
  The first assertion follows from the discussion above since $\Xi_X|_{|\rho|}$ is $C_1\epsilon$-Lipschitz with some constant 
  $C_1$ which depends only on the dimension of $\rho$. In fact, $C_1$ is the Lipschitz constant of the map $\Xi_X$ defined above.
  
  For the second assertion, let $E\to\mathcal S(X)$ be an $\epsilon$-flat bundle, and let $\rho$ be a simplex of $X$. Then
  $|\rho|$ is the geometric realization of a contractible sub-complex of $\mathcal S(X)$. Thus, by the trivialization lemma, there
  is a global $C(\rho)\epsilon$-trivialization $\Phi_\rho$ on $|\rho|$. Note that the constant $C(\rho)$ in fact only depends on
  the dimension of the simplex, so there is a constant $C'$ which only depends on the dimension of $X$, such that $C(\rho)\leq C'$
  for every simplex $\rho$. 

  Now consider another simplex $\sigma\subset\rho$, and simplices $\sigma'\subset\rho'$ of $\mathcal S(X)$ with
  $|\sigma'|\subset|\sigma|$ and $|\rho'|\subset|\rho|$. Then, for every $x\in|\rho'|$, we have that
  \[
    \Psi_{\sigma\subset\rho}(x)=[\Phi_\sigma(x,\cdot)^{-1}\Phi_{\sigma'}(x,\cdot)]
    \cdot[\Phi_{\sigma'}(x,\cdot)^{-1}\Phi_{\rho'}(x,\cdot)]\cdot[\Phi_{\rho'}(x,\cdot)^{-1}\Phi_{\rho}(x,\cdot)].
  \]
  This means that $\Psi_{\sigma\subset\rho}$ is (locally) the product of three maps $|\sigma'|\to U(\mathcal L_A(V))$ which are
  $\max(C',1)\epsilon$-Lipschitz. Now the claim follows by a two-fold application of lemma \ref{lem:product-of-lipschitz-maps}.
\end{proof}

\subsection{Extensions of almost flat bundles}

We give another application of the trivialization lemma, which will be the key observation for most of this paper.
Namely, we show that one may extend $\epsilon$-flat bundles in an essentially unique way.

Consider a simplicial complex $X$, and an $\epsilon$-flat bundle $E\to|X^{(k)}|$ over the $k$-skeleton of $X$. We want to extend
$E$ to a $C\epsilon$-small bundle $E\to|X^{(k+1)}|$ over the $(k+1)$-skeleton of $X$. Such an extension is the same thing as 
a global $C\epsilon$-trivialization over the boundary $|\partial\rho|$ of every simplex $\rho\in X_{k+1}$. Now if $k\geq 2$, then
$|\partial\rho|\cong S^k$ is simply connected, so that the trivialization lemma (or rather: corollary
\ref{korl:trivialization-lemma-simply-connected}) provides a constant $C$ such that global
$C\epsilon$-trivializations over all $|\partial\rho|$ exist if $\epsilon$ is small enough. If $k=0$, then there are trivially 
global $0$-trivializations over all $|\partial\rho|$. Thus, we have:

\begin{thm}\label{thm:extension-of-almost-flat-bundles-1}
  For every $k\in\N-\{1\}$, there are constants $C=C(k),\delta=\delta(k)>0$ with the property that for every simplicial complex 
  $X$ and every $\epsilon$-flat bundle $E\to|X^{(k)}|$ with $\epsilon\leq\delta$, there exists an extension of $E$ to a
  $C\epsilon$-flat bundle over $|X^{(k+1)}|$.\qed
\end{thm}

For $k=1$, the existence of such an extension implies that parallel transport along the boundary of every simplex is small. In
turn, this is also a sufficient condition for the existence of an extension by the trivialization lemma. This gives the following
statement.

\begin{thm}\label{thm:extensions-of-almost-flat-bundles-2}
  There are constants $C,\delta>0$ such that the following holds: Let $E\to|X^{(1)}|$ be a bundle with trivializations
  $\Phi_\sigma\colon|\sigma|\times V\to E$. Assume that transport along the boundary of every 2-simplex of $X$ is
  $\epsilon$-close to the identity where $\epsilon\leq\delta$. Then there is an extension of $E$ to a $C\epsilon$-flat bundle over
  $|X^{(2)}|$.\qed
\end{thm}

On the other hand, such extensions are unique in the following sense.

\begin{thm}\label{thm:uniqueness-of-extensions-of-almost-flat-bundles}
  Let $X$ be a simplicial complex of dimension $n$. Then there are constants $C=C(n),\delta=\delta(n)>0$, depending only on 
  the dimension of $X$, such that the following holds: Let $E,E'\to|X|$ be two $\epsilon$-small bundles modeled on the same 
  projective Hilbert $A$-module $A$, where $\epsilon\leq\delta$. We denote the trivializations by $\Phi_\rho,\Phi_\rho'$, 
  respectively. Assume that for every 1-simplex $\{p,q\}\in X_1$, we have that
  \[
    \|\Phi_q(q,\cdot)^{-1}T_{[p,q]}\Phi_p(p,\cdot)-\Phi_q'(q,\cdot)^{-1}T_{[p,q]}'\Phi_p'(p,\cdot)\|<\epsilon.
  \]
  Then there is an isomorphism of bundles $\Xi\colon E\to E'$ with the property that the map
  \[
    |\sigma|\to U(\mathcal L_A(V)),\quad x\mapsto \Phi_\rho'(x,\cdot)^{-1}\Xi\Phi_\rho(x,\cdot)
  \]
  is $C\epsilon$-Lipschitz for every simplex $\sigma\in X$.
\end{thm}
\begin{proof}
  We prove the statement by induction on the dimension $n$. We begin with the base case $n=1$. For every vertex $v\in X_0$, we
  let
  \[
    \Xi_v=\Phi_v'(v,\cdot)\Phi_v(v,\cdot)^{-1}\colon E_v\to E_v'.
  \]
  If $\rho=\{p,q\}\in X_1$ is an edge, then
  \begin{align*}
    \|&\Phi_\rho'(p,\cdot)^{-1}\Xi_p\Phi_\rho(p,\cdot)-\Phi_\rho'(q,\cdot)^{-1}\Xi_q\Phi_\rho(q,\cdot)\|\\
      &=\|\Phi_q'(q,\cdot)^{-1}\Phi_\rho'(q,\cdot)\Phi_\rho'(p,\cdot)^{-1}\Phi_p'(p,\cdot)-
          \Phi_q(q,\cdot)^{-1}\Phi_\rho(q,\cdot)\Phi_\rho(p,\cdot)^{-1}\Phi_p(p,\cdot)\|\\
          &=\|\Phi_q'(q,\cdot)^{-1}T_{[p,q]}'\Phi_p'(p,\cdot)-\Phi_q(q,\cdot)^{-1}T_{[p,q]}\Phi_p(p,\cdot)\|<\epsilon
  \end{align*}
  by assumption.

  Now we may use lemma \ref{lem:extension-of-unitary-valued-lipschitz-maps} to construct a $C(1)\epsilon$-Lipschitz map
  $f\colon|\rho|\to U(\mathcal L_A(V))$ extending the map $|\partial\rho|\to U(\mathcal L_A(V))$,
  $x\mapsto\Phi_\rho'(x,\cdot)^{-1}\Xi_x\Phi_\rho(x,\cdot)$, provided that $\epsilon$ is small enough. Obviously, the map
  \[
    \Xi_x=\Phi_\rho'(x,\cdot)f(x)\Phi_\rho(x,\cdot)^{-1}\colon E_x\to E_x'
  \]
  extends the maps $\Xi_v$ defined before, and they fit together to an isomorphism of bundles $\Xi_1\colon E|_{X^{(1)}}\to
  E'|_{X^{(1)}}$ which satisfies the Lipschitz property stated in the theorem.

  Thus, consider $n\geq 2$. We assume that we already have a bundle isomorphism 
  $\Xi_{n-1}\colon E|_{X^{(n-1)}}\to E'|_{X^{(n-1)}}$ satisfying the
  demanded Lipschitz property. Let $\rho\in X_n$ be a simplex. By assumption, the map
  \[
    |\partial\rho|\to U(\mathcal L_A(V)),\quad x\mapsto\Phi_\rho'(x,\cdot)^{-1}\Xi\Phi_\rho(x,\cdot)
  \]
  is locally (thus, globally) $C(n-1)\epsilon$-Lipschitz, and if $C(n-1)\epsilon$ is small enough, we may use lemma
  \ref{lem:extension-of-unitary-valued-lipschitz-maps} again to produce a $C(n)\epsilon$-Lipschitz extension $f\colon|\rho|\to
  U(\mathcal L_A(V))$. As before, we may put $\Xi_x=\Phi_\rho'(x,\cdot)f(x)\Phi_\rho(x,\cdot)^{-1}$, which gives the desired
  extension of $\Xi_{n-1}$ to a bundle isomorphism $\Xi_n$.
\end{proof}

\section{Almost flat bundles over Riemannian manifolds}

\label{sec:riemannian-manifolds}

In this section, let $X$ be a triangulated closed Riemannian manifold, and let $E\to X$ be an $\epsilon$-flat bundle. We 
want to show that $E$ may be equipped with a smooth structure, a smooth Hermitian metric, and a compatible connection $\nabla$, 
such that the curvature $\mathcal R^\nabla$ satisfies $\|\mathcal R^\nabla\|\leq C\epsilon$ for some constant $C>0$ which depends 
only on $X$ and on the choice of triangulation. Thus, for triangulated Riemannian manifolds, our definition of an almost flat
bundle strongly corresponds to the definition via the curvature tensor.

Here the idea is that one may define the connection inductively over neighborhoods of the skeleta of $X$. Those neighborhoods will
be constructed as subsets of the union of the open stars of the skeleta in the barycentric subdivision of $X$. The trivialization
lemma will give a trivialization over those open stars, by the following lemma.

\begin{lem}\label{lem:open-star-lemma}
  Let $X$ be a simplicial complex, let $\sigma\in X_k$ be a simplex, and consider the barycentric subdivision $\mathcal S(X)$
  of $X$. Then there is a contractible subcomplex $S\subset\mathcal S(X)$, such that $|S|$ is a neighborhood of $|\sigma|$.
\end{lem}
\begin{proof}
  Let $S_0\subset X$ be the set of those simplices $\rho$ which intersect $\sigma$. Since the vertices of $\mathcal S(X)$ are
  in a bijective correspondence with the simplices of $X$, we may define $S\subset\mathcal S(X)$ to be the subcomplex consisting
  of those simplices whose vertices lie in $S_0$.

  Then $|S|$ deformation retracts onto $|\sigma|$: Namely, if $\rho\in S_0$ is a vertex of $S$, then the mapping $t\mapsto
  t(\rho\cap\sigma)+(1-t)\rho$ is a well-defined linear curve in $S$, since by definition of $S_0$ we have that
  $\rho\cap\sigma\neq\emptyset$. This map may be extended linearly to a homotopy $|S|\times I\to|S|$ from the identity to a
  retraction onto $|\sigma|$ because the realization of a simplex $\{\sigma_0,\ldots,\sigma_k\}$ is contained in $|\sigma|$ if and
  only if all $\sigma_i\subset\sigma$. This implies that $|S|$ is contractible.

  It is clear that $|S|$ is a neighborhood of $|\sigma|$, since for every vertex $\rho\in S_0$ with $|\rho|\in|\sigma|$, the open
  star of $\rho$ is also contained in $S_0$.
\end{proof}

Now the first step in constructing the smooth bundle with connection is the following:

\begin{prop}\label{prop:open-sets-transition-functions}
  Let $X$ be a closed triangulated Riemannian manifold. Then there is a constant $C>0$, depending only on $X$, and a family of 
  open sets $U_\sigma$, one for every simplex $\sigma\in X$, such that $|\sigma|\subset U_\sigma$, with the following property:

  Let $E\to|X|$ be an $\epsilon$-flat bundle. Then there are trivializations $\Theta_\sigma\colon U_\sigma\times V\to
  E|_{U_\sigma}$, such that the transition functions
  \[
    \Psi_{\sigma,\rho}\colon U_\sigma\cap U_\rho\to U(\mathcal L_A(V)),\quad x\mapsto\Phi_\sigma(x,\cdot)^{-1}\Phi_\rho(x,\cdot)
  \]
  are all smooth and $C\epsilon$-Lipschitz.
\end{prop}
\begin{proof}
  By proposition \ref{prop:subdivision}, we may assume that $E$ is an $C_1\epsilon$-flat bundle with respect to $\mathcal S(X)$.
  Let $\sigma$ be any simplex, and consider the subcomplex $S\subset\mathcal S(X)$ of the barycentric subdivision which was
  described in lemma \ref{lem:open-star-lemma}. Then, over $|S|$, we have a global $C_2\epsilon$-trivialization $\Theta_\sigma$
  of $E$ by theorem \ref{thm:trivialization lemma}. Thus, we may define $U_\sigma$ to be the interior of $|S|$. The Lipschitzness
  is a consequence of lemma \ref{lem:product-of-lipschitz-maps}.

  However, up to now, there is no reason for the transition functions to be smooth. However, if we replace the sets $U_\sigma$ by
  smaller open subsets, it is possible to smoothen the transition functions using the following lemma.

  \begin{lem}\label{lem:smooth-approximation}
    Let $U_0,U_1,\ldots,U_k\subset X$ be open subsets, and let $\Phi_i\colon U_i\times V\to E|_{U_i}$ be trivializations such 
    that the transition functions $\Psi_{i,j}\colon U_i\cap U_j\to U(\mathcal L_A(V))$,
    $x\mapsto\Phi_i(x,\cdot)^{-1}\Phi_j(x,\cdot)$ are all $\lambda$-Lipschitz. Let $K_i\subset U_i$ be compact subsets which are
    completely contained in the image of some chart. Suppose further that all $\Psi_{i,j}$ are smooth 
    if $i,j\geq 1$. 
    
    Then there is a constant $C>0$, depending on the sets $U_i$ but not on $E$ nor $V$, and open subsets $V_0,\ldots,V_k\subset X$ 
    satisfying $K_i\subset V_i\subset U_i$, and a trivialization $\tilde\Phi_0\colon V_0\times V\to E|_{V_0}$ such that the 
    transition functions $\tilde\Psi_{0,j}\colon V_0\cap V_j\to U(\mathcal L_A(V))$, 
    $x\mapsto\tilde\Phi_0(x,\cdot)^{-1}\Phi_j(x,\cdot)$ are all smooth and $C\lambda$-Lipschitz.
  \end{lem}

  Now we may complete the proof of proposition \ref{prop:open-sets-transition-functions}. Namely, since the simplices of $X$ are
  smoothly embedded, they are all contained in a coordinate chart, so we may apply the lemma iteratively to get smooth transition
  functions which are still Lipschitz with controlled Lipschitz constant.
\end{proof}

\begin{proof}[Proof of lemma \ref{lem:smooth-approximation}]
  Inductively, we may assume that $\Psi_{0,i}$ is already smooth if $i<j$. Restricting $U_0$ to the image of a bi-Lipschitz
  chart around $K_0$, we may consider $U_0$ to be a subset of $\R^n$ with the induced metric. Let $V_0$ be an open neighborhood of
  $K_0$ in $U_0$ such that $\overline V_0\subset U_0$. For every natural number $k\in\N$, consider smooth functions 
  $\phi_k\colon\R^n\to\R_{\geq 0}$ having support in the $k^{-1}$-ball around the origin, such that $\int_{\R^k}\phi_k=1$. Then,
  if $k$ is large enough, the map
  \[
    \Psi_{0,j}'\colon V_0\cap U_j\to\mathcal L_A(V),\quad x\mapsto\int_{\R^n}\Psi_{0,j}(y)\phi_k(y-x)\,dy
  \]
  is a well-defined smooth map, and it is easily seen to have the same Lipschitz constant $\lambda$ as $\Psi_{0,j}$. 
  Furthermore, $\|\Psi_{0,j}'-\Psi_{0,j}\|_{\text{sup}}\leq\frac\lambda k$. We choose $k$ so large that $\frac\lambda k<\frac 13$.

  Now choose an open neighborhood $V_j\subset U_j$ of $K_j$ such that $\overline V_j\subset U_j$, and let $\chi\colon V_0\to[0,1]$
  be a smooth map satisfying $\chi|_{V_0\cap V_j}=1$ and $\chi=0$ on a neighborhood of $V_0-U_j$. Let
  \[
    \Psi_{0,j}''\colon V_0\cap V_j\to\mathcal L_A(V),\quad x\mapsto\chi(x)\Psi_{0,j}'(x)+(1-\chi(x))\Psi_{0,j}(x).
  \]
  Then we have that
  \begin{align*}
    \|\Psi_{0,j}''(x)-\Psi_{0,j}''(y)\|&\leq\|\chi(x)\|(\|\Psi_{0,j}'(x)-\Psi_{0,j}'(y)\|+\|\Psi_{0,j}(y)-\Psi_{0,j}(x)\|)\\
                                       &\phantom{=}+\|\chi(x)-\chi(y)\|\|\Psi_{0,j}'(y)-\Psi_{0,j}(y)\|,
  \end{align*}
  so we see that $\Psi_{0,j}''$ is $C_1\lambda$-Lipschitz for $C_1=2+L(\chi)k^{-1}$. Finally, we set 
  $\tilde\Psi_{0,j}(x)=(\Psi_{0,j}''(x)\Psi_{0,j}''(x)^*)^{-1/2}\Psi_{0,j}''(x)$. Because of lemma
  \ref{lem:projection-onto-unitaries-lipschitz}, $\tilde\Psi_{0,j}$ is
  $C\lambda$-Lipschitz for some constant which does not depend on $E$ nor on $V$.

  Now let $\tilde\Phi_0(x,\cdot):=\tilde\Psi_{0,j}(x)^{-1}\Phi_j(x,\cdot)$ for every $x\in V_0\cap U_j$, and
  $\tilde\Phi_0(x,\cdot):=\Phi_0(x,\cdot)$ outside of $U_j$.
\end{proof}

Now the idea is to construct a connection with small curvature inductively on open neighborhoods of the skeleta of $X$. In order
to extend this connection, we will need the following extension lemma for connections mainly due to Fukumoto
\cite{fukumoto-invariance-karea-surgery}.

\begin{lem}\label{lem:extension-of-almost-flat-connections}
  Let $B$ be a (not necessarily closed) Riemannian manifold without boundary. Write $X=B\times[0,5]$ equipped with the product
  metric. Then there is a constant $C>0$ with the following property:

  Let $E=X\times V$ be a trivial Hilbert module bundle over $X$, and let $E$ be equipped with a connection $\nabla$ which is
  compatible with the canonical Hermitian metric on $E$. We may write $\nabla=\partial+\Gamma$, where $\partial$ denotes taking
  directional derivatives and where $\Gamma$ is a section of $T^*X\otimes\End(E)$. Assume that $\|\Gamma_x\|\leq\epsilon\|x\|$ for
  every $x\in TX$, and that $\|\mathcal R^\nabla\|\leq\epsilon$, where $\epsilon\leq 1$.

  Then there exists another compatible connection $\tilde\nabla$ on $E$ with the following properties:
  \begin{itemize}
    \item $\|\tilde\Gamma_x\|\leq C\epsilon\|x\|$ for every $x\in TX$ if $\tilde\nabla=\partial+\tilde\Gamma$.
    \item $\|\mathcal R^{\tilde\nabla}\|\leq C\epsilon$,
    \item $\tilde\nabla=\nabla$ over $B\times[0,1]$, and
    \item $\tilde\nabla=\partial$ over $B\times[4,5]$.
  \end{itemize}
\end{lem}
\begin{proof}
	Let $\chi_1\colon[0,5]\to[0,5]$ be a smooth map satisfying $\chi_1|_{[0,1]}=\mathbbm 1$ and $\chi_1|_{[2,5]}=5$. Consider the
  map $\Phi=\id\times\chi_1\colon X\to X$. We define a new connection $\nabla'=\partial+\Gamma'$ by $\Gamma_v'=\Gamma_{\Phi_*v}$
  for every $v\in TX$. In particular, $\|\Gamma_v'\|=\|\Gamma_{\Phi_*v}\|\leq\epsilon\|\Phi_*v\|\leq c\epsilon\|v\|$ where
  $c=\|(\id\times\chi_1)_*\|_{\mathrm{sup}}\leq\sqrt{1+\|\dot\chi_1\|_{\mathrm{sup}}^2}$. Of course, $\nabla'=\nabla$ over 
  $B\times[0,1]$, but now $\Gamma_{(x,\tau)}'=\Gamma_{(x,0)}$ if $x\in TB$, $\tau\in T_t[2,5]$.

  It is easy to calculate that
  \begin{equation}\label{eq:curvature-christoffel-symbols}
		\mathcal R^{\nabla}(x,y)=(\partial_x\Gamma_y-\partial_y\Gamma_x)+[\Gamma_x,\Gamma_y]
  \end{equation}
  if $x$ and $y$ are vector fields on $X$ such that $[x,y]=0$. In particular, by a calculation in local coordinates (or
  alternatively by the fact that $\nabla'$ equals the pullback connection $\Phi^*\nabla$), this easily
  implies that $\mathcal R^{\nabla'}=\Phi^*\mathcal R^\nabla$, and therefore $\|\mathcal R^{\nabla'}\|\leq c'\epsilon$ for another
  constant $c'$. Thus, we may replace $\nabla$ by $\nabla'$, and in particular assume that $\Gamma_{(x,\tau)}=\Gamma_{(x,0)}$ over
  $B\times[2,5]$.

	Let $\chi_2\colon[2,5]\to[0,1]$ be a smooth map which satisfies $\chi_2|_{[2,3]}=1$ and $\chi_2|_{[4,5]}=0$. We let
	\[
		\tilde\Gamma_{(x,\tau)}=\chi_2(t)\Gamma_{(x,0)}
	\]
	if $\tau\in T_t[2,5]$. Then $\tilde\Gamma=0$ on $B\times[4,5]$ (and therefore $\tilde\nabla=\partial$ over $B\times[4,5]$ if we
	let $\tilde\nabla=\partial+\tilde\Gamma$), and $\tilde\Gamma_{(x,\tau)}=\Gamma_{(x,\tau)}$ if $\tau\in T_t[2,3]$, so 
	$\tilde\Gamma$ may be extended to $B\times[0,5]$ by letting $\tilde\Gamma=\Gamma$ on $B\times[0,2]$. 
	Obviously, still $\|\tilde\Gamma_x\|\leq\epsilon\|x\|$. It only remains to show that 
  $\|\mathcal R^{\tilde\nabla}\|\leq C\epsilon$.

	Now let $q=(t,p)$, $x=(v_1,\tau_1)$ and $y=(v_2,\tau_2)$ with $v_1,v_2\in T_pB$ and $\tau_1,\tau_2\in T_t[0,5]$. 
  Assume in addition that $\|x\|,\|y\|\leq 1$. Then \eqref{eq:curvature-christoffel-symbols} becomes
  \[
		\mathcal R^{\tilde\nabla}(x,y)=(\partial_{\tau_1}\chi_2)\Gamma_y-(\partial_{\tau_2}\chi_2)\Gamma_x
		+\chi_2\mathcal R^\nabla(x,y)+(\chi_2^2-\chi_2)[\Gamma_x,\Gamma_y].
  \]

	We have $\|\Gamma_x\|,\|\Gamma_y\|\leq\epsilon$, $\chi_2\leq 1$, $|\chi_2^2-\chi_2|\leq\frac 14$, $\|\mathcal
  R^\nabla(x,y)\|\leq\epsilon$ and $\partial_{\tau_i}\chi_2\leq\|\dot\chi_2\|_{\mathrm{sup}}$. Therefore,
	\[
    \|\mathcal R^{\tilde\nabla}(x,y)\|\leq 2\|\dot\chi_2\|_{\mathrm{sup}}\epsilon+\epsilon+\frac 12\epsilon^2\leq C\epsilon
	\]
  if $C=2\|\dot\chi_2\|_{\mathrm{sup}}+\frac 32$. Then $\|\mathcal R^{\tilde\nabla}\|\leq C\epsilon$ because every vector
  $\alpha\in\Lambda^2TM$ with $\|\alpha\|\leq 1$ can be written as $\alpha=x\wedge y$ with $\|x\|,\|y\|\leq 1$.
\end{proof}

We have gathered all the technical details to prove the main theorem of this section.

\begin{thm}\label{thm:almost-flat-bundles-have-small-curvature}
	Let $X$ be a smoothly triangulated Riemannian manifold. Then there are constants $C,\delta>0$ such that every $\epsilon$-flat 
  bundle $E\to X$ with $\epsilon\leq\delta$ there exists a compatible connection $\nabla$ on $E$ satisfying 
  $\|\mathcal R^\nabla\|\leq C\epsilon$.
\end{thm}
\begin{proof}
  Let $U_\sigma$, $\Theta_\sigma$ and $\Psi_{\sigma,\rho}$ be as in proposition \ref{prop:open-sets-transition-functions}.
	Assume that $\nabla$ has
	already been constructed on a neighborhood $N$ of the $(k-1)$-skeleton of $X$. We will also assume that if we write
  $\nabla=\partial+\Gamma$ with respect to any of the trivializations $\Theta_\sigma$, then $\|\Gamma_x\|\leq C_{k-1}\epsilon$ for 
  all $x\in TU_\sigma$. We have to construct $\nabla$ on a neighborhood of the $k$-skeleton.

	Let $X_k$ denote the union of all the interiors of $k$-simplices of $X$. This is a smooth submanifold, so it is the zero section
  in a tubular neighborhood $X_k\times\R^{n-k}\subset X$. Now let $\sigma$ be a simplex of dimension $k$, and let
	$T_\sigma:=\mathring\Delta^\sigma\times\R^{n-k}$. Note that for distinct $\sigma,\rho$ of dimension $k$, $T_\sigma$ and $T_\rho$ 
	are disjoint. We may assume that $T_\sigma\subset U_\sigma$. Since $N$ is a neighborhood of the
	$(k-1)$-skeleton, we can identify $\mathring\Delta^\sigma$ with $\R^k$ in such a way that $\{x\in\R^k:\|x\|\geq
  1\}\times\R^{n-k}$ is contained in $N$.

	Now $\{x\in\R^k:1\leq\|x\|\leq 2\}\times\{x\in\R^{n-k}:\|x\|\leq 1\}$ is canonically diffeomorphic to 
  $([1,2]\times S^{k-1})\times\{x\in\R^{n-k}:\|x\|\leq 1\}$ and we may assume that this space has the product metric with respect 
  to the latter product decomposition since any two metrics are in bi-Lipschitz correspondence on a compact set.
  Now by lemma \ref{lem:extension-of-almost-flat-connections} we may assume that $\nabla=\partial$ on a neighborhood of
  $\{1\}\times S^{k-1}\times\{x\in\R^{n-k}:\|x\|<1\}$ with respect to $\Theta_\sigma$, and therefore can be extended by 
  $\nabla=\partial$ on $\{x\in\R^k:\|x\|\leq 1\}\times\{x\in\R^{n-k}:\|x\|<1\}$. By 
  \ref{lem:extension-of-almost-flat-connections}, the induction 
  hypotheses are fulfilled for the new connection, and since the $T_\sigma$ are all disjoint, this construction may be performed 
  for all $k$-simplices $\sigma$ simultaneously. If now $\rho$ is another simplex such that $U_\rho$ and $U_\sigma$ intersect,
  we may write $\nabla=\partial+\Gamma^\rho$ with respect to $\Theta_\rho$. Using the $C\epsilon$-Lipschitzness of 
  $\Psi_{\sigma,\rho}$ and the fact that $\Gamma^\rho_X=\partial_X\Psi_{\sigma,\rho}+\Gamma_X\Psi_{\sigma,\rho}$ and that 
  $\Psi$ is unitary, it follows that $\|\Gamma^\rho_X\|\leq C_k\epsilon\|X\|$.
\end{proof}

\section{Pullbacks of almost flat bundles}

\label{sec:pullbacks}

In this section, we investigate pullbacks of almost flat bundles. It is rather easy to see that pullbacks of almost flat bundles
are still almost flat. However, if a map induces an isomorphism on fundamental groups, there is also a sort of converse for this
statement which we will prove. The following statement asserts that almost flat bundles are pulled back to almost flat bundles. 

\begin{prop}\label{prop:pullback-of-almost-flat-bundles}
  Let $f\colon X\to Y$ be a continuous map between simplicial complexes, and suppose that $X$ is finite-dimensional.
  Then there are constants $C,\delta>0$ such that for all $\epsilon$-flat bundles $E\to|Y|$ with $\epsilon\leq\delta$ the bundle
  $f^*E\to|X|$ admits trivializations making it an $C\epsilon$-flat bundle.
\end{prop}
\begin{proof}
  By proposition \ref{prop:subdivision}, we may replace $X$ by a repeated barycentric subdivision of itself, and therefore we may
  assume that $f$ is simplicial. Now the first statement immediately follows by pulling back the trivializations over the
  simplices of $Y$, because a simplicial map is 1-Lipschitz on every simplex.
\end{proof}

In particular, the property of being an almost flat bundle does not depend very much on the choice of triangulation, as wee see if
we take $f=\id$ in proposition \ref{prop:pullback-of-almost-flat-bundles}:

\begin{korl}
  Suppose $X$ and $X'$ are finite-dimensional triangulations of the same space, i.~e. $|X|\cong|X'|$. Then there are constants 
  $C,\delta>0$ such that for every $\epsilon\leq\delta$, every $\epsilon$-flat bundle with respect to $X$ admits the structure of 
  a $C\epsilon$-flat bundle with respect to $X'$.\qed
\end{korl}

Conversely, if $f$ induces an isomorphism of fundamental groups, and if $X$ and $Y$ are finite, then every almost flat class over 
$|X|$ is pulled back from an almost flat bundle over $|X|$. This clarifies a point left open in 
\cite{mishchenko-teleman-almost-flat-bundles}. In fact, we are going to prove a slightly stronger statement. 

For the formulation of this statement, we give a few auxiliary definitions. Recall from \cite{brunnbauer-hanke-large-and-small}
that a \emph{$\pi_1$-surjective subcomplex} of a simplicial complex $Y$ is a subcomplex $X\subset Y$ such that the inclusion
induces an epimorphism $\pi_1(X,x_0)\to\pi_1(Y,x_0)$ for every vertex $x_0\in X_0$. Denote by $\Omega_{X\subset Y}$ the set of all 
simplicial loops in $X$ which
are nullhomotopic as loops in $Y$. Consider a map $c\colon\Omega_{X\subset Y}\to\R_{>0}$. Now a \emph{$(c,\epsilon)$-flat bundle}
is an $\epsilon$-flat bundle $E\to|X|$ such that $\|T_\Gamma-\id\|\leq c(\Gamma)\epsilon$ for every 
$\Gamma\in\Omega_{X\subset Y}$. We shall prove a special case first.

\begin{lem}\label{lem:extension-of-bundles-pi1-surjective}
  Let $X\subset Y$ be a finite connected $\pi_1$-surjective subcomplex, let $X'\subset Y$ be another subcomplex containing $X$,
  and let $c\colon\Omega_{X\subset Y}\to\R_{>0}$ be a map as described above. Then there are constants $C,\delta>0$ and a map 
  $c'\colon\Omega_{X'\subset Y}\to\R_{>0}$ such that every $(c,\epsilon)$-flat bundle $E\to|X|$ is isomorphic to the restriction 
  of a $(c',C\epsilon)$-flat bundle $E'\to|X'|$ provided that $\epsilon\leq\delta$.
\end{lem}
\begin{proof}
  Suppose first that $X'$ arises from $X$ by adding a single vertex $p$ and an edge $\sigma=\{p,q\}$ where $q\in X_0$. In this
  case, we may choose arbitrary trivializations over $p$ and over $\sigma$. If $\Gamma\in\Omega_{X'\subset Y}$, transport along
  $\Gamma$ is the same thing as transport along the curve $\Gamma'$ which arises from $\Gamma$ by elimination of all occurences of 
  the piece $(q,p,q)$. Thus, we may set $c'(\Gamma)=c(\Gamma')$, $C=1$ and $\delta$ arbitrary.

  This shows that we may assume that $X'$ and $X$ have the same set of vertices. By induction on the number of simplices in
  $X'-X$, we may further assume that $X'$ arises from $X$ by adding a single simplex $\sigma$.

  If $\sigma=\{p,q\}$ is a 1-simplex, consider a simplicial path $\Gamma$ from $p$ to $q$ in $X$. Then $\Gamma'=\Gamma*(q,p)$ is a
  simplicial loop based at $p$. Since $X\subset Y$ is $\pi_1$-surjective, there is another simplicial loop $\Gamma''$, based
  at $p$, which is contained in $X$ and which is homotopic to $\Gamma'$ in $Y$. In particular, the path $\bar\Gamma''*\Gamma'$ is
  nullhomotopic in $Y$, where $\bar\Gamma''$ is the path which traverses $\Gamma''$ in the opposite direction. Thus, we may assume
  that already $\Gamma'$ is nullhomotopic in $Y$. We choose an extension of the bundle $E$, and a trivialization over $\sigma$
  such that transport along $\Gamma'$ equals the identity map. In particular, transport along $\Gamma$ and along the path $(p,q)$
  are equal. Now suppose $\tilde\Gamma$ is an arbitrary loop in $X'$ which is contractible in $Y$. Denote by $\tilde\Gamma'$ the
  path which arises from $\tilde\Gamma$ by substituting every occurence of $(p,q)$ by $\Gamma$, and every occurence of $(q,p)$ by
  the opposite $\Gamma'$. Since $(p,q)$ and $\Gamma$ are homotopic in $Y$, also $\tilde\Gamma'$ is nullhomotopic in $Y$, and 
  transport along $\tilde\Gamma'$ and along $\tilde\Gamma$ are equal. Thus, the statement of the lemma follows with
  $c'(\tilde\Gamma)=c(\tilde\Gamma')$, $C=1$ and $\delta$ arbitrary.

  If the dimension of $\sigma$ is 2, the statement follows using theorem \ref{thm:extensions-of-almost-flat-bundles-2}. If the
  dimension of $\sigma$ is at least 3, we may use theorem \ref{thm:extension-of-almost-flat-bundles-1} to get the desired
  conclusion.
\end{proof}

We are now able to prove the statement in full generality.

\begin{thm}\label{thm:functoriality-almost-flat-bundles}
  Let $X\subset Y$ and $X'\subset Y'$ be finite connected $\pi_1$-surjective subcomplexes, and let $c\colon\Omega_{X\subset
  Y}\to\R_{>0}$ be a map. Suppose $f\colon Y\to Y'$ is a map which induces an isomorphism on fundamental groups, and where
  $f(X)\subset X'$. Then there are constants $C,\delta>0$ and a map $c'\colon\Omega_{X'\subset Y'}\to\R_{>0}$ such that every
  $(c,\epsilon)$-flat bundle $E\to|X|$ is isomorphic to the pullback of a $(c',C\epsilon)$-flat bundle $E'\to|X'|$ along $f|_X$
  provided that $\epsilon\leq\delta$.
\end{thm}
\begin{proof}
  Passing to a subdivision of $X$, we may assume that $f|_X$ is simplicial. Consider the mapping cylinder $M=X\times[0,1]
  \sqcup Y'/\sim$, where $\sim$ is the equivalence relation generated by the identification $(x,1)\sim f(x)$. Now $M$ contains
  $Y'$ as a deformation retract, and the composition $X\subset M\to Y'$ equals $f|_X$. This implies that $X\times\{0\}$ is a
  $\pi_1$-surjective subcomplex of $M$, and that $\Omega_{X\subset Y}=\Omega_{X\subset M}$. On the other hand, also
  $M'=X\times[0,1]\sqcup X'/\sim$ is a finite connected $\pi_1$-surjective subcomplex of $M$ which contains both $X\times\{0\}$ 
  and $X'$. Thus, the statement of the theorem follows from lemma \ref{lem:extension-of-bundles-pi1-surjective}, applied to
  $X\times\{0\}$ and $M'$ as subcomplexes of $M$.
\end{proof}

\section{Homological invariance of infinite K-area}

\label{sec:homological-invariance}

In this chapter, we will explain how to generalize the concept of infinite K-area of closed Riemannian manifolds to homology
classes of simplicial complexes. The obvious definition is that a homology class $\eta\in H_*(X;G)$ is to be of infinite K-area if
for every $\epsilon>0$ there are $\epsilon$-flat bundles whose Chern classes detect $\eta$. However, in view of the finiteness
assumption of theorem \ref{thm:functoriality-almost-flat-bundles} it turns out that it is more useful to consider bundles which
are defined only on finite $\pi_1$-surjective subcomplexes.

\begin{dfn}
  Let $X$ be any simplicial complex with finitely generated fundamental group, and let $\eta\in H_{2*}(X;G)$. Consider a 
  finite connected $\pi_1$-surjective subcomplex $S\subset X$. Suppose 
  that there is a class $\eta_S\in H_{2*}(S;G)$ such that $\eta=\iota_*\eta_S$. Now $\eta$ is said to have \emph{infinite K-area} 
  if there is a function $c\colon\Omega_{S\subset X}\to\R_{>0}$ such that for every $\epsilon>0$ there is an $(c,\epsilon)$-flat 
  Hermitian bundle $E\to S$ such that if $f\colon S\to BU$ classifies the bundle $E$, then $f_*\eta_S\neq 0\in H_n(BU;G)$. 
\end{dfn}

Suppose $S$ and $S'$ are two different finite connected $\pi_1$-surjective subcomplexes and that there are classes $\eta_S\in
H_{2*}(S;G)$ and $\eta_{S'}\in H_{2*}(S';G)$ both mapping to $\eta$. Then there is a larger finite connected $\pi_1$-surjective
subcomplex $T$ containing both $S$ and $S'$, such that $\eta_S$ and $\eta_{S'}$ map to the same class in $H_{2*}(T;G)$. Thus,
lemma \ref{lem:extension-of-bundles-pi1-surjective} immediately implies that the definition of infinite K-area is independent of
the choice of $S$.

We now have the following immediate consequences of proposition \ref{prop:pullback-of-almost-flat-bundles} and theorem
\ref{thm:functoriality-almost-flat-bundles}.

\begin{thm}\label{thm:functoriality-infinite-karea}
  Let $f\colon X\to Y$ be a map between simplicial complexes, and consider $\eta\in H_{2*}(X;G)$.
  \begin{itemize}
    \item If $f_*\eta$ has infinite K-area, then so has $\eta$,
    \item If $\eta$ has infinite K-area and $f_*\colon\pi_1X\to\pi_1Y$ is an isomorphism, then also $f_*\eta$ has infinite K-area.
      \qed
  \end{itemize}
\end{thm}

\begin{thm}
  An even-dimensional closed oriented manifold $M^{2n}$ has infinite K-area if and only if its fundamental class 
  $[M]\in H_{2n}(M;\Q)$ has infinite K-area.
\end{thm}
\begin{proof}
  This follows directly from theorem \ref{thm:bundles-with-connection-are-epsilon-small} and theorem
  \ref{thm:almost-flat-bundles-have-small-curvature}.
\end{proof}

\begin{korl}
  A closed oriented Riemannian manifold $M$ has infinite K-area if and only if $\phi_*[M]\in H_{2n}(B\pi_1(M);\Q)$ has infinite 
  K-area, where $\phi\colon M\to\pi_1(M)$ is the classifying map of the universal bundle.\qed
\end{korl}

\begin{dfn}
  A closed oriented manifold $M$ is called \emph{essential} if $\phi_*[M]\neq 0$, where $\phi\colon M\to\pi_1M$ classifies the
  universal bundle. 
\end{dfn}

\begin{korl}
  A closed oriented manifold of infinite K-area is essential.\qed
\end{korl}

Finally, we may reprove the theorem of Fukumoto on the invariance of infinite K-area under surgery. Let $M^n$ be a differentiable
manifold. If $S^p\times D^q\subset M$ is an embedding, we consider the manifold
\[
  M^\#=(M-S^p\times\mathrm{int}(D^q))\cup_{S^p\times S^{q-1}}D^{p+1}\times S^{q-1}.
\]
We say that $M^\#$ is obtained from $M$ by \emph{$p$-surgery}.

\begin{thm}[{\cite{fukumoto-invariance-karea-surgery}}]
  Let $M^{2n}$ be a closed oriented manifold with infinite K-area, and let $M^\#$ be obtained from $M$ by surgery of index $p\neq
  1$. Then also $M^\#$ has infinite K-area.
\end{thm}
\begin{proof}
  Consider the \emph{trace}
  \[
    B=M\times I\cup_{S^p\times D^q\times\{1\}}D^{p+1}\times D^q.
  \]
  This is a bordism between $M$ and $M^\#$, so $[M]$ and $[M^\#]$ define the same class in $B$. Now let $f\colon M\to B\pi_1(M)$
  be the classifying map of the universal cover of $M$. By theorem \ref{thm:functoriality-infinite-karea}, $f_*[M]$ has infinite
  K-area. However, since $p\neq 1$, we have that $f|_{S^p\times D^q}$ is null-homotopic because $\pi_p(B\pi_1(M))=0$. Thus, $f$
  can be extended to $B$, and $f_*[M^\#]=f_*[M]$ has infinite K-area. Thus, by theorem \ref{thm:functoriality-infinite-karea},
  also $[M^\#]$ has infinite K-area.
\end{proof}

Listing \cite{listing-homology-karea} gave the following definition of infinite K-area for homology classes of manifolds $M$. A 
class $\eta\in H_{2*}(M;G)$ has infinite K-area if for every $\epsilon>0$ there exists a smooth Hermitian vector bundle $E\to M$
with compatible connection $\nabla$ such that $\|\mathcal R^\nabla\|\leq\epsilon$ and $f_*\eta\neq 0$ if $f\colon M\to BU$
classifies the bundle $E$. It is clear that our definition generalizes the definition of Listing. In the case $G=\Q$ the condition
$f_*\eta\neq 0$ simply means that some polynomial in the Chern classes of $E$ detects $\eta$.

One could obviously change this definition by demanding that a particular polynomial in the Chern classes, for instance the Chern
character, should detect $\eta$, which corresponds to $f_*\eta$ lying in a particular vector subspace of $H_{2*}(BU;\Q)$. All
statements in this section hold equally well for this kind of definition.

On the other hand, one could also consider K-homology classes $\eta\in K_0(M)$, as was done by Hanke
\cite{hanke-positive-scalar-curvature}. Here the condition on the bundles would simply be that their class pairs non-trivially
with $\eta$. Furthermore, in this case one could consider arbitrary Hilbert $A$-module bundles $E\to M$. In this case, their index
$\langle [E],\eta\rangle$ would be an element of $K_0(A)$, and one could still demand it to be nonzero. All theorems in this
section hold equally well for this definition of infinite K-area.

Finally, one could consider classes of \emph{finite K-area} as in \cite{listing-homology-karea}. Here, the K-area of a class
$\eta$ would be the largest number $a\in[0,\infty]$ such that there is a function $c$ as above with the property that for every 
$\epsilon>a^{-1}$ there is a $(c,\epsilon)$-flat bundle detecting $\eta$. Of course, proposition
\ref{prop:pullback-of-almost-flat-bundles} and theorem \ref{thm:functoriality-almost-flat-bundles} imply that there are
appropriate generalizations of theorem \ref{thm:functoriality-infinite-karea}. However, since this notion of K-area strongly
depends on the choice of triangulation, it is not clear how this might be of any use.

\section{Almost representations and quasi-representations}

\label{sec:almost-representations}

Let $X$ be a simplicial complex with finitely presented fundamental group. In this section, we will exhibit the relation between 
so-called almost representations of $\pi_1X$ and almost flat bundles over $X$. Specifically, we will show that an
$\epsilon$-almost representation of $\pi_1X$ gives a $C\epsilon$-flat bundle over $X$ and vice versa. While similar statements
have already been shown in \cite{carrion-dadarlat-almost-flat-k-theory}, and this relation has already been suggested in
\cite{connes-gromov-moscovici-conjecture-de-novikov}, it will follow easily from the ideas developped in this paper.

\subsection{Almost representations of finitely presented groups}

Recall that a group $\Pi$ is \emph{generated} by a set $L\subset\Pi$ if every element of $\Pi$ can be written as a product of
elements of $L$ and their inverses. Here we view the identity element of $\Pi$ as the empty product. 
Given such a \emph{generating set} $L\subset\Pi$, we may form the free group $\Fr(L)$
generated by the elements of $L$. Then there is a natural surjective group homomorphism $\pi\colon\Fr(L)\to\Pi$ induced by the
inclusion map $L\subset\Pi$.

Now a set of \emph{relations} is a subset $R\subset\Fr(L)$ such that the kernel of $\pi$ is the smallest normal subgroup of
$\Fr(L)$ which contains $R$. Thus, elements of $R$ are words in $L\cup L^{-1}$. A \emph{presentation} of a group $\Pi$ is a choice 
of such sets $L$ and $R$. In this situation, we write $\Pi=\langle L\mid R\rangle$. This means that every element of $\Pi$ may be 
written as a product of elements of $L\cup L^{-1}$. Such a presentation is called \emph{finite} if both $L$ and $R$ are finite.

\begin{ex}\label{ex:finite-presentation-of-finite-simplicial-complex}
  Consider a simplicial complex $X$, and let $T\subset X$ be a maximal tree in $X$. Then there is the
  following presentation of $\pi_1(X)$:
  
  For every edge $\sigma=\{v_0,v_1\}\in X_1-T$, we choose a loop $\Gamma_\sigma=\Gamma_1*(v_0,v_1)*\Gamma_2$, where $\Gamma_1$ and 
  $\Gamma_2$ are completely contained in $T$. Now the set of generators $L$ consists of the homotopy classes of the
  $\Gamma_\sigma$ for all $\sigma\in X_1-T$. The set of relations is indexed by the two simplices $\rho\in X_2$, and implements
  the fact that a curve along the boundary of $\rho$ is null-homotopic. For instance, if $\rho=\{v_0,v_1,v_2\}$ and neither of the
  edges of $\rho$ is contained in $T$, then the relation associated to $\rho$ is
  $[\Gamma_{\{v_0,v_1\}}][\Gamma_{\{v_1,v_2\}}][\Gamma_{\{v_2,v_0\}}]$. Note that this is a finite presentation if $X$ is 
  finite.
\end{ex}

\begin{dfn}
  A (unitary) \emph{$\epsilon$-almost representation} \cite{manuilov-mishchenko-almost-asymptotic-fredholm-representations} of 
  $\Pi$ on the Hilbert $A$-module $V$ with respect to the presentation 
  $\Pi=\langle L\mid R\rangle$ is a group homomorphism $\phi\colon\Fr(L)\to U(\mathcal L_A(V))$ with the 
  property that $\|\phi(r)-\id\|<\epsilon$ for every $r\in R$. We denote the set of such $\epsilon$-almost representations by
  $R_\epsilon(L\mid R)$.

  Two almost representations $\phi,\psi\colon\Fr(L)\to U(\mathcal L_A(V))$ are
  \emph{$\delta$-close} if $\|\phi(g)-\psi(g)\|\leq\delta$ for all $g\in L\subset\Fr(L)$. 
\end{dfn}

The following proposition lists a few elementary properties of almost representations.

\begin{prop}\label{prop:properties-of-almost-representations}
  Let $\Pi=\langle L\mid R\rangle$ and $\Pi'=\langle L'\mid R'\rangle$ be two finite presentations of groups, and let
  $f\colon\Pi'\to\Pi$ be a group homomorphism. Denote the canonical projections by $\pi\colon\Fr(L)\to\Pi$ and 
  $\pi'\colon\Fr(L')\to\Pi'$.
  \begin{enumerate}
    \item Let $s\colon\Fr(L')\to\Fr(L)$ be a homeomorphism satisfying $\pi\circ s=f\circ\pi'$. Then there 
      is a constant $C_1>0$, depending on the presentations and the choice of section, such that 
      $\phi\circ s\in R_{C_1\epsilon}(L'\mid R')$ whenever $\phi\in R_\epsilon(L\mid R)$.
    \item In the same situation, there is a constant $C_2>0$, such that $\phi\circ s$ and $\psi\circ s$ are
      $C_2\delta$-close whenever $\phi,\psi\in R_\epsilon(L\mid R)$ are $\delta$-close.
    \item If $s_1,s_2\colon L'\to\Fr(L)$ satisfy $\pi\circ s_1=f\circ\pi'=\pi\circ s_2$, then there is a constant $C_3>0$, such 
      that the almost representations $\phi\circ s_1$ and $\phi\circ s_2$ are $C_3\epsilon$-close whenever 
      $\phi\in R_\epsilon(L\mid R)$.
    \item Suppose that $f$ is an isomorphism. If $s\colon L'\to\Fr(L)$ and $s'\colon L\to\Fr(L')$ are such that $\pi\circ
      s=f\circ\pi'$ and $\pi'\circ s=f\circ\pi$, then there is a constant $C_4>0$, such that $\phi$ is $C_4\epsilon$-close to 
      $\phi\circ s\circ s'$ whenever $\phi\in R_\epsilon(L\mid R)$.
  \end{enumerate}
\end{prop}
\begin{proof}
  \begin{enumerate}
    \item For the first statement, note that $\pi\circ s=f\circ\pi'$ implies that $s(R')\subset\ker\pi$. Since $R'$ is finite, 
      there is a number $N\in\N$ such that every element of $s(R')$ can be written as a product of at most $N$ conjugates of 
      elements of $R\cup R^{-1}$. Thus, if $r\in R'$, there are elements $r_1,\ldots,r_k\in R\cup R^{-1}$ and 
      $w_1,\ldots,w_k\in\Fr(L)$ such that $s(r)=(w_1^{-1}r_1w_1)\cdots(w_k^{-1}r_kw_k)$, and therefore
      $\|\phi\circ s(r)-\id\|=\|(\phi(w_1)^{-1}\phi(r_1)\phi(w_1))\cdots(\phi(w_k)^{-1}\phi(r_k)\phi(w_k))-\id\|\leq C_1\epsilon$
      by lemma \ref{lem:very-useful-lemma}, where $C_1=C_1(N)$ depends only on the maximum number of factors needed.
    \item Consider $g\in L'$, and let $s(g)=g_1\cdots g_n$ where each $g_i\in L\cup L^{-1}$. Write $\phi_i=\phi(g_i)$,
      $\psi_i=\psi(g_i)$. Then, by assumption, $\|\phi_i-\psi_i\|\leq\delta$, $\phi\circ s(g)=\phi_1\cdots\phi_n$, and $\psi\circ
      s(g)=\psi_1\cdots\psi_n$. We have to show that $\|\phi_1\cdots\phi_n-\psi_1\cdots\psi_n\|\leq C\delta.$
      By induction, we only have to consider the case where $n=2$, where the claim follows from lemma \ref{lem:very-useful-lemma}
      because $\|\phi_1\phi_2-\psi_1\psi_2\|=\|(\psi_1^{-1}\phi_1)(\phi_2\psi_2^{-1})-\id\|$.
    \item Note that $\phi(s_1(g))^{-1}\phi(s_2(g))\in\ker\pi$ for every $g\in L'$, and proceed as above to show that
      $\|\phi(s_1(g))-\phi(s_2(g))\|=\|\phi(s_1(g))^{-1}\phi(s_2(g))-\id\|\leq C_3\epsilon$.
    \item Since $(s\circ s'(g))^{-1}g\in\ker\pi$, the same argument as above shows that $\|\phi(s\circ
      s'(g))-\phi(g)\|=\|\phi((s\circ s'(g))^{-1}g)-\id\|\leq C_4\epsilon$ for every $g\in L$.\qedhere
  \end{enumerate}
\end{proof}

\begin{dfn}[{\cite{manuilov-mishchenko-almost-asymptotic-fredholm-representations}}]
  Let $\Pi=\langle L\mid R\rangle$ be a finitely presented group, and $A$ a C*-algebra. An \emph{$A$-asymptotic representation} of 
  $\Pi$ with respect to this presentation is a series $\phi=(\phi_n\colon\Fr(L)\to U(\mathcal L_A(V_n)))_{n\in\N}$, such that:
  \begin{itemize}
    \item every $V_n$ is a projective sub-module of some $A^k$,
    \item for every $\epsilon>0$, there is a number $N\in\N$, such that $\phi_n$ is an $\epsilon$-almost representation whenever
      $n\geq N$,
    \item for every $\delta>0$, there is a number $N\in\N$, such that $\phi_n$ and $\phi_m$ are $\delta$-close whenever $n,m\geq
      N$, where $\phi_n$ and $\phi_m$ are considered to have values in some large enough $U(\mathcal L_A(A^k))$.
  \end{itemize}

  Two asymptotic representations $\phi=(\phi_n)$ and $\psi=(\psi_n)$ are \emph{equivalent} if for every $\delta>0$, there is a 
  number $N\in\N$, such that $\phi_n$ and $\psi_m$ are $\delta$-close whenever $n,m\geq N$.

  We denote by $R_{\mathrm{as}}(L\mid R;A)$ the set of equivalence classes of $A$-asymptotic representations.
\end{dfn}

Now proposition \ref{prop:properties-of-almost-representations} immediately implies the following:

\begin{prop}\label{prop:functoriality-of-asymptotic-representations}
  \begin{enumerate}
    \item Suppose $\Pi$ has two finite presentations $\Pi=\langle L\mid R\rangle$ and $\Pi=\langle L'\mid R'\rangle$. Then the
      sets $R_{\mathrm{as}}(L\mid R;A)$ and $R_{\mathrm{as}}(L'\mid R';A)$ are in canonical 1-to-1-correspondence. We will simply
      write $R_{\mathrm{as}}(\Pi;A)$ for any choice of finite presentation.
    \item Every group homomorphism $f\colon\Pi\to\Pi'$ induces a map $R_{\mathrm{as}}(\Pi';A)\to R_{\mathrm{as}}(\Pi;A)$.\qed
  \end{enumerate}
\end{prop}

\subsection{Almost representations and almost flat bundles}

Let $X$ be a simplicial complex with finitely presented fundamental group $\Pi=\pi_1(|X|,x_0)=\langle L\mid R\rangle$. Choose 
representing simplicial loops $\Gamma_g$ for every element $g=[\Gamma_g]\in L$. Then every $r\in R$ is
a word in $L\cup L^{-1}$, so these choices associate to $r$ a contractible simplicial loop $P_r$.

\begin{prop}
  Suppose that $E\to|X|$ is an $\epsilon$-flat bundle. Then transport along the curves $\Gamma_g$ gives a $C\epsilon$-almost 
  representation of $\pi_1(|X|,x_0)$, where $C$ is a constant depending only on $X$, the presentation, and the choices of the 
  $\Gamma_g$'s. 
\end{prop}
\begin{proof}
  Apply theorem \ref{thm:transport-along-contractible-loops} to the curves $P_r$. Since there are only finitely many of them, the
  constant from the theorem may be chosen for all $P_r$ simultaneously.
\end{proof}
  
We are going to prove that the reverse also holds true, i.~e. an $\epsilon$-almost representation of $\pi_1(|X|,x_0)$ with 
respect to the given presentation induces a $C\epsilon$-small bundle $E\to|X|$ (for another constant $C$) such that 
transport along the curves $\Gamma_g$ induces an $\epsilon$-almost representation which is close to the one we started with.

\begin{thm}\label{thm:extension-of-almost-representations}
  Let $X$ be a finite simplicial complex, and choose a finite presentation $\pi_1(|X|,x_0)=\langle L\mid R\rangle$ of the 
  fundamental group of $X$. In addition, choose representing curves $\Gamma_g$ for the generators $g=[\Gamma_g]\in L$. Then there 
  are constants $\delta,C>0$, depending on $X$, the presentation of the fundamental group, and the choices of the representing 
  curves, such that the following holds:

  Suppose $\phi\colon\Fr(G)\to U(\mathcal L_A(V))$ is an $\epsilon$-almost representation of $\pi_1(|X|,x_0)$ where
  $\epsilon\leq\delta$. Then there exists a $C\epsilon$-flat bundle $E\to|X|$ with the property that transport along the 
  curves $\Gamma_g$ gives an almost representation which is $C\epsilon$-close to $\phi$. 
\end{thm}
\begin{proof}
  We will first restrict to the case where the presentation $\pi_1(|X|,x_0)=\langle L_0\mid R_0\rangle$ is 
  the one described in example \ref{ex:finite-presentation-of-finite-simplicial-complex}, and we assume that the representing 
  curves are precisely the loops
  $\Gamma_e$ described there. Now we take $E|_{X^{(1)}}$ to be the trivial bundle $X^{(1)}\times V\to X^{(1)}$, and let the
  trivializations $\Phi_\rho\colon|\rho|\times V\to E|_{|\rho|}$ be the identities $(x,v)\mapsto(x,v)$ if $\rho\in T$ is contained
  in the maximal tree. Now we may trivialize $E|_{|e|}$ over every edge $e\in X_1-T$ such that transport along $e$ equals
  $\phi([\Gamma_e])$. We may extend this bundle with trivializations to a $C\epsilon$-flat bundle $E\to|X|$ using theorems
  \ref{thm:extension-of-almost-flat-bundles-1} and \ref{thm:extensions-of-almost-flat-bundles-2}. In turn, this extended
  bundle obviously induces the almost representation $\phi$.

  Next we want to reduce the general case to the one described above. Thus, consider an arbitrary finite presentation 
  $\pi_1(|X|,x_0)=\langle L\mid R\rangle$, and simplicial loops $\Gamma_g$ associated to the elements $g\in L$. We may now choose
  a homomorphism $s_0\colon\Fr(L_0)\to\Fr(L)$ such that $\pi s_0=\pi_0$. This defines a $C_0\epsilon$-almost representation 
  $\phi\circ s_0\colon\Fr(L_0)\to U(\mathcal L_A(V))$ by proposition \ref{prop:properties-of-almost-representations}. Now we may 
  construct the bundle $E\to|X|$ as above, for the representation $\phi\circ\tilde s_0$.
  
  We have to show that transport along the curves $\Gamma_g$ gives an almost representation which is close to $\phi$. Every 
  $\Gamma_g$ is of the form
  \[
    \Gamma_g=(\Gamma_g^0,I_{e_1},\Gamma_g^1,I_{e_2},\ldots,I_{e_{k_g}},\Gamma_g^{k_g})
  \]
  where each $e_i$ is an edge in $X_1-T$, and each $\Gamma_g^i$ is completely contained in $T$. This defines a homomorphism
  $s\colon\Fr(L)\to\Fr(L_0)$ via $g\mapsto[\Gamma_{e_{k_g}}]\cdots[\Gamma_{e_1}]$. Now transport along $\Gamma_g$ equals transport 
  along the compositions of the curves $\Gamma_{e_{k_g}}$, i.~e., it equals $\phi\circ s_0\circ s(g)$. Thus, transport along the
  curves $\Gamma_g$ gives the almost representation $\phi\circ s_0\circ s$, which is $C\epsilon$-close to $\phi$ by
  proposition \ref{prop:properties-of-almost-representations}.
\end{proof}

It also turns out that the isomorphism class of the bundle is uniquely determined by the almost representation induced by the
transport, as the following theorem shows.

\begin{thm}\label{thm:uniqueness-of-extension-of-almost-representations}
  Let $X$ be a finite simplicial complex, and let $\pi_1(|X|,x_0)=\langle L\mid R\rangle$ be a finite presentation of the
  fundamental group of $X$. Suppose that we have two choices $\Gamma_g,\tilde\Gamma_g$ for the generators in $L$. Then there is a 
  constant $\delta>0$, depending on $X$, the presentation, and the choices $\Gamma_g$ and $\Gamma_g'$ such that the 
  following holds:

  If $E\to|X|$ and $E'\to X$ are $\delta$-flat bundles such that transport in $E$ along the 
  curves $\Gamma_g$ and transport in $E'$ along the curves $\Gamma_g'$ give almost representations which are 
  $\delta$-close, then $E$ and $E'$ are isomorphic bundles.
\end{thm}
\begin{proof}
  As in example \ref{ex:finite-presentation-of-finite-simplicial-complex}, we choose a maximal tree $T\subset X$. Then for every
  vertex $v\in X$, we choose a simplicial path from $x_0$ to $v$ which is completely contained in $T$, and trivialize $E$ and $E'$
  over $v$ by composing parallel transport along the chosen path with the given trivializations over $x_0$. This trivialization
  does not depend on the choice of path in $T$, and parallel transport along edges in $T$ becomes trivial with respect to this
  choice of trivialization.

  Again, we consider the standard presentation $\pi_1(|X|,x_0)=\langle L_0\mid R_0\rangle$ from example
  \ref{ex:finite-presentation-of-finite-simplicial-complex}. As in the proof of theorem
  \ref{thm:extension-of-almost-representations}, parallel transport along the edges in $X_1-T$ gives almost presentations
  $\phi_0,\phi_0'\colon\Fr(L_0)\to U(\mathcal L_A(V))$. The choices of $\Gamma_g$ and $\Gamma_g'$ induce group homomorphisms
  $s,s'\colon\Fr(L) \to\Fr(L_0)$, such that $\phi_0\circ s$ and $\phi_0'\circ s'$ are the almost presentations given by parallel 
  transporting along the curves $\Gamma_g$ and $\Gamma_g'$, respectively.

  Choose a homomorphism $s\colon\Fr(L_0)\to\Fr(L)$ satisfying $\pi s=\pi_0$. By assumption,
  $\phi\circ s$ and $\phi\circ s'$ are $\delta$-close, so that $\phi\circ\tilde s\circ\tilde s_0$ and
  $\phi\circ\tilde s\circ\tilde s_0$ are $C_0\delta$-close by proposition \ref{prop:properties-of-almost-representations}. On
  the other hand, these almost representations are $C_1\delta$-close to $\phi_0$ and $\phi_0'$, respectively, again by proposition 
  \ref{prop:properties-of-almost-representations}. This implies that $\phi_0$ and $\phi_0'$ are $C_2\delta$-close.
  However, this $C_2\delta$-closeness is precisely the condition for theorem
  \ref{thm:uniqueness-of-extensions-of-almost-flat-bundles} to work, so the bundles are isomorphic if $\delta$ is small enough.
\end{proof}

\subsection{Asymptotically flat K-theory}

A class $\eta\in K^0(X;A)$ can be represented as the difference $\eta=[E_1]-[E_2]$ of two Hilbert $A$-module bundles $E_i\to|X|$.
We denote by $K^0_\epsilon(X;A)\subset K^0(X;A)$ the set of those classes such that $E_1$ and $E_2$ may be chosen to be
$\epsilon$-flat. In addition, we define the subset of \emph{asymptotically flat K-theory classes} by 
$K^0_{\mathrm{af}}(X;A)=\bigcap_{\epsilon>0}K^0_\epsilon(X;A)$. That is, a class $\eta\in K^0(X;A)$ is asymptotically flat if for
every $\epsilon>0$, there exist $\epsilon$-flat Hilbert $A$-module bundles $E_1,E_2\to|X|$ such that $\eta=[E_1]-[E_2]$.

Note that there is an obvious notion of direct sum for asymptotic representations, which makes $R_{\mathrm{as}}(\pi_1X;A)$ into a
semi-group. Now theorems \ref{thm:extension-of-almost-representations} and
\ref{thm:uniqueness-of-extension-of-almost-representations} show that there is a well-defined semi-group homomorphism
$R_{\mathrm{as}}(\pi_1X;A)\to K^0_{\mathrm{af}}(X;A)$ which induces a group homomorphism
$\alpha\colon\mathrm{Gr}(R_{\mathrm{as}}(\pi_1X;A))\to K^0_{\mathrm{af}}(X;A)$. By proposition
\ref{prop:functoriality-of-asymptotic-representations}, $\alpha$ is surjective. Furthermore, one can show that $\alpha$ is
compatible with the pullback maps of asymptotic representations and asymptotically flat K-theory, so it gives a natural
transformation. However, $\alpha$ is certainly not an isomorphism ($R_{\mathrm{as}}(\pi_1X;A)$ is not even abelian), so it would
be interesting to examine the kernel of $\alpha$.

\appendix

\section{Parallel transport and curvature}

\label{app:parallel-transport-curvature}

In this section, we will prove proposition \ref{prop:transport-along-small-curves-is-small} which states that parallel transport
along curves which bound a small area is small. This proof follows ideas from \cite{rani-on-parallel-transport} and an 
unpublished proof by Jost-Hinrich Eschenburg, who in turn learned the idea from Hermann Karcher. 

In the course of the proof we will need the following lemma:

\begin{lem}\label{lem:parallel-transport-along-curves}
  Let $E\to[0,1]$ be a smooth Hilbert $A$-module bundle with (not necessarily compatible) connection $\nabla$. We denote parallel 
  transport along $\gamma$ by $T_\gamma(t)\colon E_t\to E_1$ and consider a section $s\colon[0,1]\to E$. Then
  \[
    \partial_t(T_\gamma(t)s(t))=T_\gamma(t)\nabla_{\partial_t}s(t)
  \]
  for all $t\in[0,1]$.
\end{lem}
\begin{proof}
  If $E$ is modeled on a free Hilbert $A$-module, the statement is easily shown by writing both sides in a parallel frame.
  In the general case, one has to consider another bundle $E'$ such that $E\oplus E'$ is modeled on a free Hilbert $A$-module. It
  is easily possible to extend the connection on $E$ to a connection on $E\oplus E'$, so the general case follows from the free
  case.
\end{proof}

\begin{proof}[Proof of proposition \ref{prop:transport-along-small-curves-is-small}]
  Take $x\in E_{f(0,0)}$ with $\|x\|=1$, and let $x'=P_{\partial f}x$. For $s\in[0,1]$, let $X(s,0)\in E_{f(s,0)}$ be the
  parallel translate of $x$ along the curve $s\mapsto f(s,0)$, and for $(s,t)\in[0,1]$, let $X(s,t)\in E_{f(s,t)}$ be the parallel
  translate of $X(s,0)$ along the curve $t\mapsto f(s,t)$.

  Furthermore, let $P_{(s,t)}\colon E_{f(s,t)}\to E_{f(1,1)}$ be defined by first parallel translating along $t\mapsto f(s,t)$ and
  then along $s\mapsto f(s,1)$. Now, by definition, $P_{(0,0)}x'=X(1,1)$, and $P_{(0,0)}x=P_{(0,1)}X(0,1)$, so that
  \[
    P_{(0,0)}(x'-x)=P_{(1,1)}X(1,1)-P_{(0,1)}X(0,1)=\int_0^1\partial_s\left(P_{(s,1)}X(s,1)\right)\,ds.
  \]
  Since $P_{(s,1)}$ is parallel transport along the curve $s\mapsto(s,1)$, lemma \ref{lem:parallel-transport-along-curves} implies
  that $\partial_s(P_{(s,1)}X(s,1))=P_{(s,1)}\nabla_{\partial_sf(s,1)}X(s,1)$. Now $X(s,0)$ is parallel along $s\mapsto f(s,0)$ by
  definition, so that $\nabla_{\partial_sf}X(s,0)=0$. Again with lemma \ref{lem:parallel-transport-along-curves}, it follows that
  \[
    P_{(s,1)}\nabla_{\partial_sf}X(s,1)=\int_0^1\partial_t(P_{(s,t)}\nabla_{\partial_sf}X(s,t))\,dt
    =\int_0^1P_{(s,t)}\nabla_{\partial_tf}\nabla_{\partial_sf}X(s,t)\,dt
  \]
  for all $s\in[0,1]$. In addition, we have that $\nabla_{\partial_tf}X(s,t)=0$ since $X$ is parallel in the $t$-direction by 
  definition. Therefore, $\mathcal R^\nabla(\partial_tf\wedge\partial_sf)X=\nabla_{\partial_tf}\nabla_{\partial_sf}X$.

  Since $\nabla$ is compatible with the metric, one can easily show that parallel transport preserves the metric and in particular
  the norm. Thus, the equations combine to give
  \[
    \|x'-x\|\leq\int_0^1\int_0^1\|\mathcal R^\nabla(\partial_tf\wedge\partial_sf)\|\,dt\,ds.\qedhere
  \]
\end{proof}

\section{Unitary elements of C*-algebras}

\label{app:unitary-elements-of-cstar-algebras}

In this section, we will give a proof of lemma \ref{lem:extension-of-unitary-valued-lipschitz-maps}, which states that
$\epsilon$-Lipschitz maps from the sphere $S^{n-1}$ into the unitary elements $U(A)$ of an arbitrary C*-algebra $A$ may be 
extended to $C\epsilon$-Lipschitz maps on the whole disk $D^n$ whenever
$\epsilon$ is small enough. Here $C$ is some universal constant which depends neither on the C*-algebra $A$, nor on the dimension
of the sphere $S^{n-1}$. The result will be
important even for the classical case of Hermitian vector bundles since it shows that maps into the set of unitary matrices can be
extended as above with a constant $C$ independent of the size of the matrices.

We begin with a statement which allows the extension of Lipschitz maps $S^{n-1}\to V$ to Lipschitz maps on $D^n$ if $V$ is a
normed vector space.

\begin{lem}\label{lem:extension-of-lipschitz-maps-to-vector-spaces}
  There is a universal constant $C_0>0$ with the following property:
  Let $\beta_0\colon S^{n-1}\to V$ be a $\lambda$-Lipschitz map into a normed vector space $V$. Assume additionally that
  $\beta_0(S^{n-1})\subset B_R(0)$ for a number $R>0$ and $\beta_0(s_0)=0$ for some $s_0\in S^{n-1}$. Then there is an extension
  $\beta\colon D^n\to B_R(0)\subset V$ which is Lipschitz with constant at most $C_0\lambda$.
\end{lem}
\begin{proof}
  A first idea would be to define the extension by $\beta(t\cdot x)=t\beta_0(x)$ for $x\in\partial D^n$, $t\in[0,1]$. This
  certainly gives a continuous extension, but it turns out that the problem of calculating the Lipschitz constant for the
  resulting map is not as easy as it looks. However, one can simply do the contraction on a ring and extend constantly by zero on
  the interior, i.~e.
  \[
    \beta(t\cdot x)=\begin{cases}(2t-1)\beta_0(x),&t\geq\frac 12,\\0,&t\leq\frac 12.\end{cases}
  \]
  Then, using that $\{x\in\R^n:\frac 12\leq\|x\|\leq 1\}$ and $S^{n-1}\times[0,1]$ are bi-Lipschitz equivalent (and that, by an
  explicit calculation, the Lipschitz constants do not depend on $n$), one can easily deduce the statement.
\end{proof}

We will need the statement that every function given by holomorphic functional calculus is Fréchet differentiable.

\begin{lem}\label{lem:holomorphic-functional-calculus-frechet-differentiable}
  Let $f\colon U_\epsilon\to\C$ be a holomorphic map where $U_\epsilon=\{z\in\C:\|z\|<\epsilon\}$ is a ball of radius $\epsilon>0$ 
  around $0\in\C$.
  Let $f(z)=\sum_{n=0}^\infty\lambda_nz^n$ be the power series expansion of $f$ around $0$. Furthermore, consider the power series
  $\hat f(z)=\sum_{n=0}^\infty|\lambda_n|z^n$. If $A$ is a Banach algebra, then the map
  \[
    \tilde f\colon\{x\in A:\|x\|<\epsilon\}\to A,\quad x\mapsto\sum_{n=0}^\infty\lambda_nx^n
  \]
  is well-defined and Fréchet differentiable, and the operator norm of the Fréchet differential of $\tilde f$ at $x\in A$ is
  bounded by $\hat f'(\|x\|)$. For every $x$ in the domain of $\tilde f$, we have that $x\tilde f(x)=\tilde f(x)x$. This is a
  special case of the so-called \emph{functional calculus}.
\end{lem}
\begin{proof}
  For every point $x$ in the domain of $\tilde f$, it is easy to see that the series
  \[
    \Delta_x\colon A\to A,\quad h\mapsto\sum_{n=1}^\infty\lambda_n\sum_{k=0}^{n-1}x^khx^{n-k-1}
  \]
  converges and gives a linear map with operator norm bounded by $\hat f'(\|x\|)$. Now a straightforward calculation shows that
  this is the Fréchet differential of $\tilde f$ at $x$.
\end{proof}

Now let $A$ be a C*-algebra. We write $U(A)=\{u\in A: uu^*=u^*u=1\}$, $\mathfrak u(A)=\{v\in A:v^*=-v\}$ and $H_A=\{v\in
A:v^*=v\}$. Elements of $U(A)$ are called \emph{unitary}, elements of $\mathfrak u(A)$ are called \emph{skew-Hermitian}, and
elements of $H_A$ are called \emph{Hermitian}. Obviously $A=\mathfrak u(A)\oplus H_A$ as a vector space. We denote by $\pi\colon
A\to\mathfrak u(A)$ the projection onto the first summand with respect to this decomposition

If $v\in\mathfrak u(A)$, we have that $v^2\in H_A$. Now we consider the map $f\colon\{z\in\C:\|z\|<\frac 12\}\to\C$ which is given
by $f(z)=(1+z^2)^{1/2}$. Since $f(z)=f(-z)$, $f$ is really a power series in $z^2$. so if we define $\tilde f$ as in lemma
\ref{lem:holomorphic-functional-calculus-frechet-differentiable}, we see that $\tilde f$ maps elements of $\mathfrak u(A)$ into
$H_A$. 

\begin{lem}\label{lem:parametrization-of-unitaries-is-locally-lipschitz}
  There is a constant $L>0$, independent of $A$, such that the so-defined map
  \[
    \tilde f\colon\left\{v\in\mathfrak u(A):\|v\|<\frac 12\right\}\to H_A,\quad v\mapsto (1+v^2)^{1/2}
  \]
  is Lipschitz with constant at most $L$.
\end{lem}
\begin{proof}
  By lemma \ref{lem:holomorphic-functional-calculus-frechet-differentiable}, the map $f$ is Fréchet differentiable, and 
  the operator norm of the differential is bounded by a differentiable map $\hat f'\colon(-1,1)\to\R$
  which is independent of $A$. This map is bounded by some constant $L>0$ if it is restricted to the closed interval $[0,\frac
  12]$, so the operator norm of the Fréchet differential of $f$ is bounded by $L$ on its whole domain. Now the result immediately
  follows immediately using the convexity of the domain of $f$.
\end{proof}

Now let $v\in\mathfrak u(A)$ with $\|v\|<\frac 12$ and $w=\tilde f(v)\in H_A$. Since $w$ arose from $v$ 
by functional calculus, the elements $v$ and $w$ commute, so
\[
  (v+w)(v+w)^*=(v+w)(-v+w)=-v^2+[v,w]+w^2=-v^2+(1+v^2)=1,
\]
and, similarly, $(v+w)^*(v+w)=1$. Therefore, $v+w\in U(A)$.

On the other hand, the projection $\pi\colon A\to\mathfrak u(A)$ is a linear map with operator norm equal to $1$ because
$\pi(a)=\frac 12(a-a^*)$. Thus, $\|\pi(x)\|=\|\pi(x)-\pi(1)\|=\|\pi(x-1)\|\leq\|x-1\|$ for all $x\in A$ where we used that $1\in
H_A$. Now we can prove the announced extension result for maps into $U(A)$.

\begin{proof}[Proof of lemma \ref{lem:extension-of-unitary-valued-lipschitz-maps}]
  We consider the map
  \[
    g\colon\left\{v\in\mathfrak u(A):\|v\|<\frac 12\right\}\to U(A),\quad v\mapsto v+(1+v^2)^{1/2}
  \]
  and the projection $\pi\colon U(A)\subset A=\mathfrak u(A)\oplus H_A\to\mathfrak u(A)$. Using lemma
  \ref{lem:parametrization-of-unitaries-is-locally-lipschitz}, one easily shows that $g$ is $(1+L)$-Lipschitz, where $L>0$ is
  independent of $A$. We have already noted that $\pi$ is $1$-Lipschitz. 

  Now let $\alpha_0\colon S^{n-1}\to U(A)$ be as in the statement of the theorem. We choose $s_0\in S^{n-1}$ and consider the map
  \[
    \alpha_1\colon S^{n-1}\to\mathfrak u(A),\quad x\mapsto\pi(\alpha_0(s_0)^{-1}\alpha_0(x)).
  \]
  Since $\alpha_0(s_0)^{-1}\in U(A)$ and multiplication by unitary elements is an isometry in C*-algebras, we have that
  $\alpha_1(S^{n-1})\subset\{v\in\mathfrak u(A):\|v\|<\frac 12\}$. Furthermore, $\alpha_1$ is Lipschitz with constant at most
  $\lambda$. We may now extend $\alpha_1$ to a $(C_0\lambda)$-Lipschitz map $\alpha_2\colon D^n\to\{v\in\mathfrak
  u(A):\|v\|<\frac 12\}$ using lemma \ref{lem:extension-of-lipschitz-maps-to-vector-spaces}. Now we let
  \[
    \alpha\colon D^n\to U(A),\quad x\mapsto\alpha_0(s_0)\cdot g\alpha_2(x).
  \]
  Then the Lipschitz constant of $\alpha$ is at most $C_0(1+L)\lambda$, and it is rather clear that $\alpha|_{S^{n-1}}=\alpha_0$. 
  This shows the statement of the theorem with constant $C=C_0(1+L)$, which is in fact independent of $A$.
\end{proof}

Note that the condition on $\diam(\alpha_0)$ is immediate if $\lambda\leq\frac 14$.

The methods in this chapter can also be used to prove the following statement, which is used in the proof of lemma
\ref{lem:smooth-approximation}.

\begin{lem}\label{lem:projection-onto-unitaries-lipschitz}
  The map
  \[
    f\colon\left\{x\in GL(A):\dist(x,U(A))<\frac 13\right\}\to U(A),\quad x\mapsto (xx^*)^{-1/2}x
  \]
  is well-defined, equals the identity on $U(A)$, and is Lipschitz with some Lipschitz constant $L$ which does not depend on $A$.
\end{lem}
\begin{proof}
  Consider an element $x$ in the domain of $f$. Then $\|xx^*-1\|=\|(x-u)x^*+u(x-u)^*\|\leq\|x-u\|(\|x\|+\|u\|)<\frac 13(2+\frac
  13)=\frac 79$. Because of lemma \ref{lem:holomorphic-functional-calculus-frechet-differentiable}, the map $H_A\to H_A$ sending
  $h$ to $h^{-1/2}$ is well-defined and $L_1$-Lipschitz on the set of all $h\in H_A$ with $\|h-1\|\leq\frac 79$. Using this, the
  Lipschitzness of $f$ is straightforward. The other assertions of the lemma are clear.
\end{proof}

\printbibliography

\end{document}